\RequirePackage{fix-cm}
\documentclass[smallextended]{svjour3}       
\smartqed  
\usepackage{graphicx}
\usepackage{amsmath}
\usepackage{amssymb}
\usepackage{enumerate} 
\usepackage{algorithm,algorithmic}
\usepackage{caption}
\usepackage{array}
\usepackage{pgfplots, pgfplotstable, booktabs, colortbl, array}
\usepackage{multirow}
\usepackage{float}
\floatstyle{plaintop}
\restylefloat{table}

\usepackage{array}
\usepackage{color}
\usepackage{cases}
\usepackage{hyperref}
\usepackage{cite}
\usepackage{algorithmic}
\usepackage{pifont}
\usepackage{algorithm}
\usepackage{graphicx, color}
\usepackage{epstopdf}
\usepackage{booktabs}
\usepackage{diagbox}
\usepackage{multirow, multicol}
\usepackage{arydshln}
\usepackage{subfigure}
\usepackage{amsfonts}
\usepackage[misc]{ifsym}
\graphicspath{{figs/}}

\usepackage{geometry}
\usepackage{setspace}
\usepackage[inline, shortlabels]{enumitem}

\newcommand{\norm}[1]{\|#1\|}
\newcommand{\abs}[1]{|#1|}

\newcommand{\uu}{{\bf u}}

\newcolumntype{L}[1]{>{\raggedleft\let\newline\\\arraybackslash\hspace{0pt}}m{#1}}
\newcolumntype{R}[1]{>{\raggedleft\let\newline\\\arraybackslash\hspace{0pt}}m{#1}}
\usepackage{geometry}
\usepackage{setspace}
\usepackage[inline, shortlabels]{enumitem}
\setlist[enumerate]{leftmargin=.5in}
\setlist[itemize]{leftmargin=.5in}

\begin{document}
\title{An improved Shifted CholeskyQR based on columns}
\author{Yuwei Fan \textsuperscript{1} \and Haoran Guan \textsuperscript{2} \and {Zhonghua Qiao \textsuperscript{2*}}}
\institute{%
\begin{itemize}[leftmargin=*]
  \item[*] {Zhonghua Qiao} \\
        {Corresponding author.} \\
        \email{zhonghua.qiao@polyu.edu.hk}
        \\
  \item[] {Yuwei Fan} \\
        \email{fanyuwei2@huawei.com}
        \\
  \item[] {Haoran Guan} \\
        \email{21037226R@connect.polyu.hk}
  \at
  \item[\textsuperscript{1}] Theory Lab, Huawei Hong Kong Research Center, Sha Tin, Hong Kong \\
  \item[{2}] Department of Applied Mathematics, The Hong Kong Polytechnic University, Hung Hom, Hong Kong
\end{itemize}
}
\titlerunning{An improved Shifted CholeskyQR based on columns}

\date{Received: date / Accepted: date}
\maketitle

\begin{abstract}
Among all the deterministic CholeskyQR-type algorithms, Shifted CholeskyQR3 is specifically designed to address the QR factorization of ill-conditioned matrices. This algorithm introduces a shift parameter $s$ to prevent failure during the initial Cholesky factorization step, making the choice of this parameter critical for the algorithm's effectiveness. Our goal is to identify a smaller $s$ compared to the traditional selection based on $\norm{X}_{2}$. In this research, we propose a new definition for the input matrix $X$ called $[X]_{g}$, which is based on the column properties of $X$. $[X]_{g}$ allows us to obtain a reduced shift parameter $s$ for the Shifted CholeskyQR3 algorithm, thereby improving the sufficient condition of $\kappa_{2}(X)$ for this method. We provide rigorous proofs of orthogonality and residuals for the improved algorithm using our proposed $s$. Numerical experiments confirm the enhanced numerical stability of orthogonality and residuals with the reduced $s$. We find that Shifted CholeskyQR3 can effectively handle ill-conditioned $X$ with a larger $\kappa_{2}(X)$ when using our reduced $s$ compared to the original $s$. Furthermore, we compare CPU times with other algorithms to assess performance improvements.
\end{abstract}



\section{Introduction}
\label{sec:introduction}
As a fundamental component of matrix decomposition, QR factorization plays a crucial role in various real-world applications across academia and industry. These applications include randomized singular value decomposition \cite{halko2011, 2020}, Krylov subspace methods \cite{2010}, the local optimal block preconditioned conjugate gradient method (LOBPCG) \cite{2018robust}, and block HouseholderQR algorithms \cite{1989}, among others.

\subsection{CholeskyQR2 and its properties}
Typical algorithms for QR factorization include CGS, MGS, HouseholderQR and TSQR. For details, see \cite{ballard2011, 2011, MatrixC, Higham, Communication, Numerical}. In recent years, a new algorithm called CholeskyQR has been developed. When $X \in \mathbb{R}^{m\times n}$, $m \ge n$, CholeskyQR begins by computing a Gram matrix $B \in \mathbb{R}^{n\times n}$, followed by performing a Cholesky factorization to obtain an upper-triangular matrix $R \in \mathbb{R}^{n\times n}$. The orthogonal factor $Q \in \mathbb{R}^{m\times n}$ can then be computed. For CholeskyQR, $X$ needs to be full rank, that is, $\mbox{rank(X)}=n$. Algorithm~\ref{alg:cholqr} illustrates the fundamental version of CholeskyQR that computes the QR factorization as follows.

\begin{algorithm}[H]
\caption{$[Q,R]=\mbox{CholeskyQR}(X)$}
\label{alg:cholqr}
\begin{algorithmic}[1]
\STATE $B=X^{\top}X,$
\STATE $R=\mbox{Cholesky}(B),$
\STATE $Q=XR^{-1}.$
\end{algorithmic}
\end{algorithm}%

Compared to TSQR, HouseholderQR, MGS, and CGS, the CholeskyQR algorithm has several advantages. It has only half the computational cost of TSQR and HouseholderQR. Additionally, it requires significantly fewer reductions in a parallel environment than the other algorithms. Moreover, CholeskyQR utilizes BLAS3 operations, which are more difficult to implement for other algorithms.

Although with many advantages, Algorithm~\ref{alg:cholqr} exhibits certain limitations and is rarely used directly. When considering the error of orthogonality, it is shown in \cite{error} that
\begin{equation}
\norm{Q^{\top}Q-I}_{F} \le \frac{5}{64}\delta^{2}, \label{eq:ortho}
\end{equation}
where
\begin{equation}
\delta=8\kappa_{2}(X)\sqrt{mn\uu+n(n+1)\uu}. \label{eq:delta}
\end{equation}
Here, $\kappa_{2}(X)=\frac{\sigma_{1}(X)}{\sigma_{n}(X)}$ is the condition number of $X$. $\sigma_{i}(X)$ is the $i$-th largest singular value of $X$ for $i=1,2,3,…, n$. Specifically, $\sigma_{1}(X)=\norm{X}_{2}$. $\uu$ is the machine precision and $\uu=2^{-53}$. In CholeskyQR-type algorithms, regarding the sizes of the matrices, we always set
\begin{align}
mn\uu &\le \frac{1}{64}, \label{eq:mn} \\
n(n+1)\uu &\le \frac{1}{64}. \label{eq:nn}
\end{align}

According to \eqref{eq:ortho} and \eqref{eq:delta}, the orthogonality error of Algorithm~\ref{alg:cholqr} is proportional to $(\kappa_2(X))^{2}$. Numerous numerical experiments indicate that Algorithm~\ref{alg:cholqr} is numerically stable only when the input $X$ is very well-conditioned. Consequently, a new algorithm, named CholeskyQR2, has been developed by performing two iterations of the CholeskyQR algorithm \cite{2014}. It is presented in Algorithm~\ref{alg:cholqr2}.

\begin{algorithm}[H]
\caption{$[Q_{1},R_{2}]=\mbox{CholeskyQR2}(X)$}
\label{alg:cholqr2}
\begin{algorithmic}[1]
\STATE $[Q,R]=\mbox{CholeskyQR}(X),$
\STATE $[Q_{1},R_{1}]=\mbox{CholeskyQR}(Q),$
\STATE $R_{2}=R_{1}R.$
\end{algorithmic}
\end{algorithm}%

In \cite{error}, it has been shown that compared to Algorithm~\ref{alg:cholqr}, Algorithm~\ref{alg:cholqr2} is numerically stable in both orthogonality and residual. The following lemma holds.

\begin{lemma}[Rounding error analysis of CholeskyQR2]
For $X \in \mathbb{R}^{m\times n}$ and $[Q_{1},R_{2}]=\mbox{CholeskyQR2}(X)$, when $8\kappa_{2}(X)\sqrt{mn\uu+n(n+1)\uu} \le 1$, we have
\begin{align}
\norm{Q_{1}^{\top}Q_{1}-I}_{F} &\le 6(mn\uu+n(n+1)\uu), \label{eq:qq} \\
\norm{Q_{1}R_{2}-X}_{F} &\le 5n^{2}\uu\norm{X}_{2}. \label{eq:q1r2}
\end{align}
\end{lemma}

Here, the smallness of \eqref{eq:delta} is a sufficient condition for Algorithm~\ref{alg:cholqr2}, indicating that Algorithm~\ref{alg:cholqr2} is reliable when the input $X$ is not ill-conditioned.

\subsection{Shifted CholeskyQR3 and its problems}
When $X$ is ill-conditioned, Algorithm~\ref{alg:cholqr2} may encounter numerical breakdown due to rounding errors. To address this challenge, researchers have introduced an improved algorithm known as Shifted CholeskyQR (SCholeskyQR), which is detailed in Algorithm~\ref{alg:Shifted} \cite{Shifted}.

\begin{algorithm}
\caption{$[Q,R]=\mbox{SCholeskyQR}(X)$}
\label{alg:Shifted}
\begin{algorithmic}[1]
\STATE $B=X^{\top}X,$
\STATE choose $s>0$,
\STATE $R=\mbox{Cholesky}(B+sI),$
\STATE $Q=XR^{-1}.$
\end{algorithmic}
\end{algorithm}%

Algorithm~\ref{alg:Shifted} is a superior algorithm in terms of applicability compared to Algorithm~\ref{alg:cholqr}. The concept behind the algorithm is straightforward. For an ill-conditioned matrix $B \in \mathbb{R}^{n\times n}$, the addition of a scaled identity matrix reduces $\kappa_2(B+sI)$ and prevents numerical breakdown. To further improve the numerical stability, CholeskyQR2 is performed subsequently, and a new algorithm called Shifted CholeskyQR3 (SCholeskyQR3) has been developed, which is given in Algorithm~\ref{alg:Shifted3}.

\begin{algorithm}[H]
\caption{$[Q_{2},R_{4}]=\mbox{SCholeskyQR3}(X)$}
\label{alg:Shifted3}
\begin{algorithmic}[1]
\STATE $[Q,R]=\mbox{SCholeskyQR}(X),$
\STATE $[Q_{1},R_{1}]=\mbox{CholeskyQR}(Q),$
\STATE $R_{2}=R_{1}R,$
\STATE $[Q_{2},R_{3}]=\mbox{CholeskyQR}(Q_{1}),$
\STATE $R_{4}=R_{3}R_{2}.$
\end{algorithmic}
\end{algorithm}%

Regarding Algorithm~\ref{alg:Shifted} and Algorithm~\ref{alg:Shifted3}, we have the following theoretical results from \cite{Shifted}.

\begin{lemma}[Rounding error analysis of Shifted CholeskyQR]
\label{lemma SCholeskyQR}
For $X \in \mathbb{R}^{m\times n}$ and $[Q,R]=\mbox{SCholeskyQR}(X)$, with $11(mn\uu+n(n+1)\uu)\norm{X}_{2}^{2} \le s \le \frac{1}{100}\norm{X}_{2}^{2}$ and $\kappa_{2}(X) \le \frac{1}{6n^{2}\uu}$, we have
\begin{align}
\norm{Q^{\top}Q-I}_{2} &\le 2, \label{eq:qqo} \\
\norm{QR-X}_{F} &\le 2n^{2}\uu\norm{X}_{2}. \label{eq:qro}
\end{align}
\end{lemma}

\begin{lemma}[The relationship between $\kappa_{2}(X)$ and $\kappa_{2}(Q)$ for Shifted CholeskyQR]
\label{lemma Condition numbers}
For $X \in \mathbb{R}^{m\times n}$ and $[Q,R]=\mbox{SCholeskyQR}(X)$, with $11(mn\uu+n(n+1)\uu)\norm{X}_{2}^{2} \le s \le \frac{1}{100}\norm{X}_{2}^{2}$ and $\kappa_{2}(X) \le \frac{1}{6n^{2}\uu}$, we have
\begin{equation}
\kappa_{2}(Q) \le 2\sqrt{3} \cdot \sqrt{1+\alpha(\kappa_{2}(X))^{2}}. \label{eq:qx}
\end{equation}
Here, $\alpha=\frac{s}{\norm{X}_{2}^{2}}$. When $[Q_{2},R_{4}]=\mbox{SCholeskyQR3}(X)$, if we take $s=11(mn\uu+n(n+1)\uu)\norm{X}_{2}^{2}$ and $\kappa_{2}(X)$ is large enough, a sufficient condition for $\kappa_{2}(X)$ is
\begin{equation}
\kappa_{2}(X) \le \frac{1}{96(mn\uu+n(n+1)\uu)}. \label{eq:c2}
\end{equation}
\end{lemma}

\begin{lemma}[Rounding error analysis of Shifted CholeskyQR3]
\label{lemma SCholeskyQR3}
For $X \in \mathbb{R}^{m\times n}$ and $[Q_{2},R_{4}]=\mbox{SCholeskyQR3}(X)$,  with $s=11(mn\uu+n(n+1)\uu)\norm{X}_{2}^{2}$ and \eqref{eq:c2}, we have
\begin{align}
\norm{Q_{2}^{\top}Q_{2}-I}_{F} &\le 6(mn\uu+n(n+1)\uu), \label{eq:q2q2} \\
\norm{Q_{2}R_{4}-X}_{F} &\le 15n^{2}\uu\norm{X}_{2}. \label{eq:q2r4}
\end{align}
\end{lemma}

In particular, Lemma~\ref{lemma Condition numbers} is one of the most important properties of Shifted CholeskyQR3. It shows that when $s$ is located in a certain interval, the larger $s$ we take, the larger $\kappa_{2}(Q)$ will become. Since CholeskyQR2 following Shifted CholeskyQR will break down if $\kappa_{2}(Q)$ is large, the selection of the shifted parameter $s$ is a crucial aspect of Shifted CholeskyQR3. 
It can neither be too large, considering the applicability of Shifted CholeskyQR3, nor too small, since Shifted CholeskyQR may break down.

In fact, when we express the first two steps of Algorithm~\ref{alg:Shifted} with error analysis as 
\begin{align}
B &= X^{\top}X+E_{A}, \label{eq:11} \\
R^{\top}R &= B+E_{B}+sI,  \label{eq:12}
\end{align}
we find that the error bounds of $\norm{E_{A}}_{2}$ and $\norm{E_{B}}_{2}$ in \eqref{eq:11} and \eqref{eq:12} significantly influence the choice of the parameter $s$. In \cite{Shifted}, the original $s$ is set to be $10$ times the sum of $\norm{E_{A}}_{2}$ and $\norm{E_{B}}_{2}$. Previous researchers have used $\norm{X}_{2}$ to bound the $2$-norm of each column of $X$ when estimating $\norm{E_{A}}_{2}$ and $\norm{E_{B}}_{2}$. However, in practice, both $\norm{E_{A}}_{2}$ and $\norm{E_{B}}_{2}$ tend to be overestimated. This overestimation results in a relatively large value of $s$, which gives a more stringent sufficient condition for $\kappa_{2}(X)$ in the context of Shifted CholeskyQR3, based on \eqref{eq:qx}, \eqref{eq:c2}, and the corresponding analytical steps outlined in \cite{Shifted}. This condition will limit the applicability of Shifted CholeskyQR3 to ill-conditioned matrices, as demonstrated by numerous numerical experiments. In most of the cases, the $2$-norm of each column of $X$ can be significantly smaller than $\norm{X}_{2}$.  Therefore, the primary objective of this work is to select a smaller shifted parameter $s$ for Shifted CholeskyQR3 and to demonstrate that this improved $s$ can ensure the numerical stability of the algorithm. We aim to provide a more accurate error estimation for the residuals of Shifted CholeskyQR3 theoretically. The revised choice of $s$ improves the applicability of Shifted CholeskyQR3, which is reflected in a better sufficient condition for $\kappa_{2}(X)$ to some extent.

\subsection{Our contributions in this work}
In this work, we calculate the largest $2$-norm among all the columns of $X$, which is defined as $[X]_{g}$ in Definition~\ref{def:g}.

\begin{definition}[The definition of ${[\cdot]_{g}}$]
\label{def:g}
For $X=[X_{1},X_{2}, \cdots X_{n-1},X_{n}]\in R^{m\times n}$,
\begin{equation}
[X]_{g}:=\max_{1 \le j \le n}\norm{X_{j}}_{2}, \label{eq:21}
\end{equation}
where
\begin{equation}
\norm{X_{j}}_{2}=\sqrt{x_{1,j}^{2}+x_{2,j}^{2}+……+x_{m-1,j}^{2}+x_{m,j}^{2}}. \nonumber
\end{equation}
\end{definition}

We introduce several properties of $[\cdot]_{g}$ of the matrix in Section~\ref{sec:g}, which offer a new perspective on rounding error analysis. Using $[\cdot]_{g}$, we can estimate $\norm{E_{A}}_{2}$ and $\norm{E_{B}}_{2}$ with tighter upper bounds based on $[X]_{g}$. Consequently, a smaller $s$ with $[X]_{g}$ can be chosen as $s=11(mn\uu+n(n+1)\uu)[X]_{g}^{2}$ for Shifted CholeskyQR3. Regarding $[\cdot]_{g}$, we define a constant $p$ as
\begin{equation}
p=\frac{[X]_{g}}{\norm{X}_{2}}. \label{eq:p}
\end{equation}
Here, $\frac{1}{\sqrt{n}} \le p \le 1$. We present the following theorems related to the improved Shifted CholeskyQR (ISCholeskyQR) and improved Shifted CholeskyQR3 (ISCholeskyQR3).

\begin{theorem}[Rounding error analysis of the improved Shifted CholeskyQR]
\label{thm:41}
For $X \in \mathbb{R}^{m\times n}$ and $[Q,R]=\mbox{ISCholeskyQR}(X)$, with $11(mn\uu+n(n+1)\uu)[X]_{g}^{2} \le s \le \frac{1}{100}[X]_{g}^{2}$ and $\kappa_{2}(X) \le \frac{1}{4.89pn^{2}\uu}$, we have
\begin{align}
\norm{Q^{\top}Q-I}_{2} &\le 1.6, \label{eq:227} \\
\norm{QR-X}_{F} &\le 1.67pn^{2}\uu\norm{X}_{2}. \label{eq:228}
\end{align}
\end{theorem}

\begin{theorem}[The relationship between $\kappa_{2}(X)$ and $\kappa_{2}(Q)$ for the improved Shifted CholeskyQR]
\label{thm:42}
For $X \in \mathbb{R}^{m\times n}$ and $[Q,R]=\mbox{ISCholeskyQR}(X)$, with $11(mn\uu+n(n+1)\uu)[X]_{g}^{2} \le s \le \frac{1}{100}[X]_{g}^{2}$ and $\kappa_{2}(X) \le \frac{1}{4.89pn^{2}\uu}$, we have
\begin{equation}
\kappa_{2}(Q) \le 3.24\sqrt{1+t(\kappa_{2}(X))^{2}}. \label{eq:235}
\end{equation}
When $[Q_{2},R_{4}]=\mbox{ISCholeskyQR3}(X)$, if we take $s=11(mn\uu+n(n+1)\uu)[X]_{g}^{2}$ and $\kappa_{2}(X)$ is large enough, a sufficient condition for $\kappa_{2}(X)$ is
\begin{equation}
\kappa_{2}(X) \le \frac{1}{86p(mn\uu+(n+1)n\uu)} \le \frac{1}{4.89pn^{2}\uu}. \label{eq:236}
\end{equation}
Here, we define
\begin{equation}
t=\frac{s}{\norm{X}_{2}^{2}} \le \frac{1}{100}. \label{eq:216}
\end{equation}
\end{theorem}

\begin{theorem}[Rounding error analysis of the improved Shifted CholeskyQR3]
\label{thm:43}
For $X \in \mathbb{R}^{m\times n}$ and $[Q_{2},R_{4}]=\mbox{ISCholeskyQR3}(X)$, with $s=11(mn\uu+n(n+1)\uu)[X]_{g}^{2}$ and \eqref{eq:236}, we have
\begin{align}
\norm{Q_{2}^{\top}Q_{2}-I}_{F} &\le 6(mn\uu+n(n+1)\uu), \label{eq:242} \\
\norm{Q_{2}R_{4}-X}_{F} &\le (6.57p+4.87)n^{2}\uu\norm{X}_{2}. \label{eq:243}
\end{align}
\end{theorem}

Theorems~\ref{thm:41}-~\ref{thm:43} correspond to Lemmas~\ref{lemma SCholeskyQR}-~\ref{lemma SCholeskyQR3}, respectively, which are proved in Sections~\ref{sec:45}-~\ref{sec:47}. These theorems demonstrate that the improved Shifted CholeskyQR3 has a better sufficient condition of $\kappa_{2}(X)$ compared to the original one. Consequently, the improved Shifted CholeskyQR3 can effectively handle $X$ with larger $\kappa_{2}(X)$, as shown in Table \ref{tab:Comparison} and the numerical experiments in Section \ref{sec:experiments}. The property of the $R$-factor can also be described by $[\cdot]_{g}$, which will loosen the upper bound of $\kappa_{2}(X)$ in the existing results in \cite{Shifted}. From a theoretical perspective, we prove the numerical stability of the algorithm when using an improved $s$ in Section~\ref{sec:main} and provide tighter theoretical upper bounds of the residual $\norm{Q_{2}R_{4}-X}_{F}$ using the properties of $[\cdot]_{g}$ under such cases compared to the original one in \cite{Shifted}, seeing Table~\ref{tab:Comparisonu} for detailed comparisons. This provides new insights into the problem of rounding error analysis.

\begin{table}[t!]
\caption{Comparison of $\kappa_{2}(X)$ between the improved and the original $s$}
\centering
\begin{tabular}{||c c c||}
\hline
$s$ & $\mbox{Sufficient condition of $\kappa_{2}(X)$}$ & $\mbox{Upper bound of $\kappa_{2}(X)$}$ \\
\hline
$11(mn\uu+n(n+1)\uu)\norm{X}_{2}^{2}$ & $\frac{1}{96(mn\uu+n(n+1)\uu)}$ & $\frac{1}{6n^{2}\uu}$ \\
\hline
$11(mn\uu+n(n+1)\uu)[X]_{g}^{2}$ & $\frac{1}{86p(mn\uu+n(n+1)\uu)}$ & $\frac{1}{4.89pn^{2}\uu}$ \\
\hline
\end{tabular}
\label{tab:Comparison}
\end{table}

\begin{table}[t!]
\caption{Comparison of the upper bounds between the improved and the original $s$}
\centering
\begin{tabular}{||c c c||}
\hline
$s$ & $\mbox{SCholeskyQR}$ & $\mbox{SCholeskyQR3}$\\
\hline
$11(mn\uu+n(n+1)\uu)\norm{X}_{2}^{2}$ & $2n^{2}\uu\norm{X}_{2}$ & $15n^{2}\uu\norm{X}_{2}$ \\
\hline
$11(mn\uu+n(n+1)\uu)[X]_{g}^{2}$ & $1.6n^{2}\uu[X]_{g}$ & $(6.57p+4.87)n^{2}\uu\norm{X}_{2}$ \\
\hline
\end{tabular}
\label{tab:Comparisonu}
\end{table}

Defining $[X]_{g}$ offers several advantages. In many cases, when the size of $X$ is large, \textit{e.g.}, $m>10^{5}$ or $n>10^{4}$, $X$ tends to be sparse for storage efficiency. In such scenarios, calculating the norms of the matrix can be computationally expensive. The properties of $[\cdot]_{g}$ allow us to select an $s$ based on key elements of $X$ without the need to compute the norms of the entire large matrix. Furthermore, $[\cdot]_{g}$ enables better utilization of the matrix structure and the inherent properties of its elements, while the $2$-norm primarily highlights the general characteristics of the matrix. We plan to leverage these properties for further exploration of CholeskyQR-type algorithms in our future works. In other words, the definition of $[\cdot]_{g}$ offers a novel approach to rounding error analysis for matrices, based on their structures and elements. Although this perspective is not directly evident from numerical experiments, it represents an innovative advancement compared to existing results.

\subsection{Outline and notations}
The rest of the paper is organized as follows. Section~\ref{sec:literature} reviews existing results on rounding error analysis. In Section~\ref{sec:g}, we examine some properties of $[\cdot]_{g}$. The proof of the theoretical results for the improved Shifted CholeskyQR3 is presented in Section~\ref{sec:main}, which serves as the key contribution of this work. Next, Section~\ref{sec:experiments} provides numerical results and compares them with several existing algorithms. Finally, concluding remarks are offered in Section~\ref{sec:discussions}.

In this work, $\norm{\cdot}_{F}$ and $\norm{\cdot}_{2}$ denote the Frobenius norm and the $2$-norm of the matrix. For the input matrix $X$, $\abs{X}$ is the matrix whose elements are all the absolute values of the elements of $X$.

\section{Preliminary Lemmas of rounding error analysis}
\label{sec:literature}
Before presenting our main results, we introduce the following preliminary lemmas for rounding error analysis.

\begin{lemma}[Weyl's Theorem \cite{MatrixC}]
\label{lemma 2.1}
 For matrices $A,B,C \in \mathbb{R}^{m\times n}$, if we have $A+B=C$, then
\begin{equation}
\abs{\sigma_{i}(A)-\sigma_{i}(B)} \le \sigma_{1}(C)=\norm{C}_{2}. \nonumber
\end{equation}
where $\sigma_{i}(X)$ is the $i$-th greatest singular value of $X$, with $i=1,2, \cdots, \min(m,n)$.
\end{lemma}

\begin{lemma}[Rounding error in matrix multiplications \cite{Higham}]
\label{lemma 2.2}
For $A \in \mathbb{R}^{m\times n}, B \in \mathbb{R}^{n\times p}$, the error in computing the matrix product $C=AB$ in floating-point arithmetic is bounded by
\begin{equation}
\abs{AB-fl(AB)}\le \gamma_{n}\abs{A}\abs{B}. \nonumber
\end{equation}
Here, $\abs{A}$ is the matrix whose $(i,j)$ element is $\abs{a_{ij}}$ and
\begin{equation}
\gamma_n: = \frac{n{\uu}}{1-n{\uu}} \le 1.02n{\uu}. \nonumber
\end{equation}
\end{lemma}

\begin{lemma}[Rounding error in Cholesky factorization \cite{Higham}]
\label{lemma 2.3}
When $A \in \mathbb{R}^{n\times n}$ is symmetric positive definite, its output $R$ after Cholesky factorization in floating-point arithmetic satisfies
\begin{equation}
R^{\top}R=A+\Delta{A}, \quad \abs{\Delta A}\le \gamma_{n+1}\abs{{R}^{\top}}\abs{R}. \nonumber
\end{equation}
\end{lemma}

\begin{lemma}[Rounding error in solving triangular systems \cite{Higham}]
\label{lemma 2.4}
If $R \in \mathbb{R}^{n\times n}$ is a nonsingular upper-triangular matrix, the computed solution $x$ obtained by solving an upper-triangular linear system
$Rx=b$ by back substitution in floating-point arithmetic satisfies
\begin{equation}
(R+\Delta R)x=b,  \quad \abs{\Delta R}\le \gamma_{n}\abs{R}. \nonumber
\end{equation}
\end{lemma}

To learn more about matrix perturbations, readers can refer to references \cite{Improved, Perturbation, Backward}.

\section{Some properties of $[\cdot]_{g}$}
\label{sec:g}
In this section, we introduce and prove several properties of $[\cdot]_{g}$ following Definition~\ref{def:g}. These properties will be utilized in the theoretical analysis of the improved Shifted CholeskyQR3.

\begin{lemma}[Connections between ${[\cdot]_{g}}$ and some matrix norms]
\label{lemma 2.5}
If $A \in \mathbb{R}^{m\times p}, B \in \mathbb{R}^{p\times n}$, then we have
\begin{equation}
[AB]_{g} \le \norm{A}_{2}[B]_{g}, [AB]_{g} \le \norm{A}_{F}[B]_{g}. \label{eq:A3}
\end{equation}
\end{lemma}
\begin{proof}
Regarding $[AB]_{g}$, with Definition~\ref{def:g}, we have
\begin{equation}
\begin{split}
[AB]_{g} &\le \max(\norm{AB_{1}}_{2},\norm{AB_{2}}_{2},\cdots,\norm{AB_{n}}_{2}) \nonumber \\ &\le \max(\norm{A}_{2}\norm{B_{1}}_{2},\norm{A}_{2}\norm{B_{2}}_{2},\cdots,\norm{A}_{2}\norm{B_{n}}_{2}) \nonumber \\ &\le \norm{A}_{2} \cdot \max(\norm{B_{1}}_{2},\norm{B_{2}}_{2},\cdots,\norm{B_{n}}_{2}) \nonumber \\ &\le \norm{A}_{2}[B]_{g}. \nonumber
\end{split}
\end{equation}
Here, the first inequality of \eqref{eq:A3} is received. Since $\norm{A}_{2} \le \norm{A}_{F}$, it is easy to get the second inequality of \eqref{eq:A3}.
\end{proof}

\begin{lemma}[The triangular inequality of ${[\cdot]_{g}}$]
\label{lemma 2.6}
For $A,B,C \in \mathbb{R}^{m\times n}$, when $A=B+C$, we have
\begin{equation}
[A]_{g} \le [B]_{g}+[C]_{g}. \label{eq:A8}
\end{equation}
\begin{proof}
Based on Definition~\ref{def:g} and the triangular inequality of the norms of vectors, we can easily get \eqref{eq:A8}.
\end{proof}
\end{lemma}

\begin{lemma}[The relationship between ${[\cdot]_{g}}$ and some matrix norms]
\label{lemma 2.7}
For $A \in \mathbb{R}^{m\times n}$, we have
\begin{equation}
[A]_{g} \le \norm{A}_{2} \le \norm{A}_{F}. \label{eq:A9}
\end{equation}
\begin{proof}
The left inequality is based on the property of the singular values of the matrix. The right inequality is obvious.
\end{proof}
\end{lemma}

\section{Theoretical analysis of the improved Shifted CholeskyQR3}
\label{sec:main}
In this section, we provide the theoretical analysis of the improved Shifted CholeskyQR3 with a smaller $s$. We present the relevant settings and lemmas for our algorithm, and we also prove Theorem~\ref{thm:41}, Theorem~\ref{thm:42} and Theorem~\ref{thm:43}.

\subsection{General settings and assumptions}
\label{sec:41}
Given the presence of rounding errors at each step of the algorithm, we express Algorithm~\ref{alg:Shifted} with error matrices as follows.
\begin{align}
B &= X^{\top}X+E_{A}, \label{eq:22} \\
R^{\top}R &= B+sI+E_{B}, \label{eq:23} \\
q_{i}^{\top} &= x_{i}^{\top}(R+E_{Ri})^{-1}, \label{eq:24} \\
X &= QR+E_{X}. \label{eq:25}
\end{align}
We let $q_{i}^{\top}$ and $x_{i}^{\top}$ represent the $i$-th rows of $X$ and $Q$ respectively. The error matrix $E_{A}$ in \eqref{eq:22} denotes the discrepancy generated when calculating the Gram matrix $X^{\top}X$. Similarly, $E_{B}$ in \eqref{eq:23} represents the error matrix after performing Cholesky factorization on $B$ with a shifted item. As noted in \cite{Shifted}, since $R$ may be non-invertible, we describe the last step of Algorithm~\ref{alg:Shifted} in terms of each row of the matrices in \eqref{eq:24}, where $E_{Ri}$ denotes the rounding error for the $R$ factor. If we write the last step of Algorithm~\ref{alg:Shifted} without $R^{-1}$, the general error matrix of QR factorization is given by $E_{X}$ in \eqref{eq:25}. A crucial aspect of the subsequent analysis is establishing connections between $E_{X}$ and $E_{Ri}$.

Under \eqref{eq:21}, we provide a new interval of the shifted item $s$ based on $[X]_{g}$. If $X \in \mathbb{R}^{m\times n}$, except \eqref{eq:mn} and \eqref{eq:nn}, we have the following settings.
\begin{align}
4.89n^{2}\uu \cdot p\kappa_{2}(X) &\le 1, \label{eq:28} \\
11(mn\uu+n(n+1)\uu)[X]_{g}^{2} &\le s \le \frac{1}{100} [X]_{g}^{2}. \label{eq:29}
\end{align}
Here, $p$ is defined in \eqref{eq:p}. We observe that, compared to the original Shifted CholeskyQR based on $\norm{X}_{2}$, the range of $\kappa_{2}(X)$ expands with a constant $p$ related to $n$ as indicated in \eqref{eq:28}. Furthermore, \eqref{eq:29} demonstrates that the new $s$ is still constrained by a relative large upper bound. The applicability of this new $s$ can be established using a method similar to those in \cite{Floating,Super,Modified}.

\subsection{Algorithms}
\label{sec:42}
In this section, we present the improved Shifted CholeskyQR (ISCholeskyQR) and the improved Shifted CholeskyQR3 (ISCholeskyQR3). They are detailed in Algorithm~\ref{alg:AS} and Algorithm~\ref{alg:AS3}, respectively.

\begin{algorithm}[H]
\caption{$[Q,R]=\mbox{ISCholeskyQR}(X)$}
\label{alg:AS}
\begin{algorithmic}[1]
\STATE calculate the norm of each column for $X$, $\norm{X_{j}}$, $j=1,2,3, \cdots n,$
\STATE pick $[X]_{g}=\max_{1 \le j \le n}\norm{X_{j}},$
\STATE choose $s=11(mn\uu+(n+1)n\uu)[X]_{g}^{2},$
\STATE $[Q,R]=\mbox{SCholeskyQR}(X).$
\end{algorithmic}
\end{algorithm}%

\begin{algorithm}[H]
\caption{$[Q_{2},R_{4}]=\mbox{ISCholeskyQR3}(X)$}
\label{alg:AS3}
\begin{algorithmic}[1]
\STATE calculate the norm of each column for $X$, $\norm{X_{j}}$, $j=1,2,3, \cdots n,$
\STATE pick $[X]_{g}=\max_{1 \le j \le n}\norm{X_{j}},$
\STATE choose $s=11(mn\uu+(n+1)n\uu)[X]_{g}^{2},$
\STATE $[Q,R]=\mbox{SCholeskyQR}(X),$
\STATE $[Q_{1},R_{1}]=\mbox{CholeskyQR}(Q),$
\STATE $R_{2}=R_{1}R,$
\STATE $[Q_{2},R_{3}]=\mbox{CholeskyQR}(Q_{1}),$
\STATE $R_{4}=R_{3}R_{2}.$
\end{algorithmic}
\end{algorithm}%

\subsection{Some lemmas for proving theorems}
\label{sec:44}
To prove Theorems~\ref{thm:41}-~\ref{thm:43}, we require the following lemmas. These theoretical results resemble those in \cite{Shifted} and their proofs closely follow those of \cite{Shifted}. However, by utilizing the definition of $[\cdot]_{g}$ and its properties, we can improve several bounds in the algorithm. We will discuss these improvements in detail below.

\begin{lemma}
\label{lemma 4.1}
For $E_{A}$ and $E_{B}$ in \eqref{eq:22} and \eqref{eq:23}, if \eqref{eq:29} is satisfied, we have
\begin{align}
\norm{E_{A}}_{2} &\le 1.1mn\uu[X]_{g}^{2}, \label{eq:ea} \\
\norm{E_{B}}_{2} &\le 1.1n(n+1)\uu[X]_{g}^{2}. \label{eq:eb}
\end{align}
\end{lemma}
\begin{proof}
In this section, we aim to estimate $\norm{E_{A}}_{2}$ and $\norm{E_{B}}_{2}$ using $[X]_{g}$ instead of $\norm{X}_{2}$. While our analysis follows a similar approach to that in \cite{Shifted, error}, our new definition of $[\cdot]_{g}$ allows us to provide improved estimations for the two norms of the error matrices.

We can estimate $\norm{E_{A}}_{F}$ first since $\norm{E_{A}}_{2}$ is bounded by $\norm{E_{A}}_{F}$. With Lemma~\ref{lemma 2.2} and \eqref{eq:22}, we have $B=fl(X^{\top}X)$. Therefore, we have
\begin{equation} \label{eq:ea1}
\begin{split}
\abs{E_{A}} &= \abs{B-X^{\top}X} \\ &\le \gamma_{m}\abs{X^{\top}}\abs{X}. 
\end{split}
\end{equation}
With \eqref{eq:ea1}, for $E_{Aij}$ which denotes the element of $E_{A}$ in the $i$-th row and the $j$-th column, we have
\begin{equation} 
\abs{E_{Aij}} \le \gamma_{m}\abs{X_{i}}\abs{X_{j}}. \label{eq:eqeaij}
\end{equation}
Here, $X_{i}$ denotes the $i$-th column of $X$. We combine \eqref{eq:eqeaij} with \eqref{eq:21} and have
\begin{equation}
\abs{E_{Aij}} \le \gamma_{m}[X]_{g}^{2}. \label{eq:eqeaij1}
\end{equation}
Since $\norm{E_{A}}_{F}=\sqrt{\sum_{i=1}^{n}\sum_{j=1}^{n}(\abs{E_{Aij}})^{2}}$, with \eqref{eq:eqeaij1}, we can bound $\norm{E_{A}}_{2}$ as
\begin{equation}
\begin{split}
\norm{E_{A}}_{2} \le \norm{\abs{E_{A}}}_{F} &\le \gamma_{m}\sqrt{\sum_{i=1}^{n}\sum_{j=1}^{n}(\abs{E_{Aij}})^{2}} \nonumber \\ &\le \gamma_{m}n[X]_{g}^{2} \nonumber \\ &\le 1.1mn\uu[X]_{g}^{2}. \nonumber
\end{split}
\end{equation}
Then, \eqref{eq:ea} is proved. \eqref{eq:ea} is a more accurate estimation of $\norm{E_{A}}_{2}$ compared to that in \cite{Shifted, error} since $[X]_{g} \le \norm{X}_{2}$.

When estimating $\norm{E_{B}}_{F}$, we focus on \eqref{eq:23}. We use the same idea as that in \cite{Shifted, error} for this estimation. With \eqref{eq:21}, we have
\begin{equation}
\norm{R}_{F}^{2}=\norm{\abs{R}}_{F}^{2} \le n[R]_{g}^{2}. \label{eq:rf}
\end{equation}
Using Lemma~\ref{lemma 2.3}, Lemma~\ref{lemma 2.6}, \eqref{eq:22}, \eqref{eq:23} and \eqref{eq:rf}, we can get
\begin{equation} \label{eq:214}
\begin{split}
\norm{E_{B}}_{2} \le \norm{\abs{E_{B}}}_{F} &\le \gamma_{n+1}\norm{R}_{F}^{2} \\ &\le \gamma_{n+1} \cdot n[R]_{g}^{2} \\ &\le \gamma_{n+1} \cdot n([X]_{g}^{2}+s+\norm{E_{A}}_{2}+\norm{E_{B}}_{2}). 
\end{split}
\end{equation}
With \eqref{eq:mn}, \eqref{eq:nn}, \eqref{eq:29}, \eqref{eq:ea} and \eqref{eq:214}, we can bound $\norm{E_{B}}_{2}$ as
\begin{equation}
\begin{split}
\norm{E_{B}}_{2} &\le \frac{\gamma_{n+1}n(1+\gamma_{m}n+t_{1})}{1-\gamma_{n+1}n}[X]_{g}^{2} \nonumber \\
&\le \frac{1.02(n+1){\uu}\cdot n(1+1.02m{\uu}\cdot n+0.01)}{1-1.02(n+1){\uu}\cdot n}[X]_{g}^{2} \nonumber \\
&\le \frac{1.02\cdot n(n+1){\uu}\cdot(1+1.02\cdot\frac{1}{64}+0.01)}{1-\frac{1.02}{64}}[X]_{g}^{2} \nonumber \\
&\le 1.1n(n+1){\uu}[X]_{g}^{2}. \nonumber
\end{split}
\end{equation}
\eqref{eq:eb} is proved. Here, we define $t_{1}=\frac{s}{\norm{X}_{2}^{2}} \le 0.01$ based on \eqref{eq:29}. In all, Lemma~\ref{lemma 4.1} is proved.
\end{proof}
\begin{remark}
The last step of \eqref{eq:214} relies on Lemma~\ref{lemma 2.6} and Lemma~\ref{lemma 2.7}. While the approach for estimating $\norm{E_{B}}_{2}$ parallels that in \cite{Shifted, error}, we utilize the relationships between the $2$-norm and $[\cdot]_{g}$ established in Lemma~\ref{lemma 2.7}, which derive from a distinctly different perspective on rounding error analysis compared to the original work on CholeskyQR-type algorithms.
\end{remark}

\begin{lemma}
\label{lemma 4.2}
For $R^{-1}$ and $XR^{-1}$ from \eqref{eq:24}, when \eqref{eq:29} is satisfied, we have
\begin{align}
\norm{R^{-1}}_{2} &\le \frac{1}{\sqrt{(\sigma_{n}(X))^{2}+0.9s}}, \label{eq:223} \\
\norm{XR^{-1}}_{2} &\le 1.5. \label{eq:224}
\end{align}
\end{lemma}
\begin{proof}
The steps of analysis to get \eqref{eq:223} and \eqref{eq:224} are similar to those in \cite{Shifted}. Lemma~\ref{lemma 4.2} is proved.
\end{proof}

\begin{lemma}
\label{lemma 4.3}
For $E_{Ri}$ from \eqref{eq:24}, when \eqref{eq:29} is satisfied, we have
\begin{equation}
\norm{E_{Ri}}_{2} \le 1.03n\sqrt{n}\uu[X]_{g}. \label{eq:225}
\end{equation}
\end{lemma}
\begin{proof}
The steps to get \eqref{eq:225} are similar to those in \cite{Shifted}. However, the property of $[\cdot]_{g}$ can provide a tighter bound of $\norm{E_{Ri}}_{2}$. For $1\le i\le m$, based on Lemma~\ref{lemma 2.4}, we have
\begin{equation}
\norm{E_{Ri}}_{2} \le \norm{E_{Ri}}_{F} \le \gamma_n \sqrt{n}[R]_{g}. \label{eq:eri}
\end{equation}
Based on the properties of Cholesky factorization and the structure of the algorithm, we find that the square of $[\cdot]_{g}$ of the matrix corresponds to the largest entry on the diagonal of the Gram matrix. With Lemma~\ref{lemma 2.7}, \eqref{eq:22}, \eqref{eq:23} and \eqref{eq:29}, we obtain
\begin{equation}
[R]_{g}^{2} \le [X]_{g}^{2}+s+(\norm{E_{A}}_{2}+\norm{E_{B}}_{2}) \le 1.011[X]_{g}^{2}. \label{eq:rg2}
\end{equation}
With \eqref{eq:rg2}, it is easy to see that
\begin{equation}
[R]_{g} \le 1.006[X]_{g}. \label{eq:rg}
\end{equation}
Therefore, we put \eqref{eq:rg} into \eqref{eq:eri} and we can get\eqref{eq:225}. Lemma~\ref{lemma 4.3} is proved.
\end{proof}

\begin{lemma}
\label{lemma 4.4}
For $E_{X}$ from \eqref{eq:25}, when \eqref{eq:29} is satisfied, we have
\begin{equation}
\norm{E_{X}}_{2} \le \frac{1.15n^{2}\uu[X]_{g}^{2}}{\sqrt{(\sigma_{n}(X))^{2}+0.9s}}. \label{eq:226}
\end{equation}
\end{lemma}
\begin{proof}
For Shifted CholeskyQR, $R$ will not always be invertible due to errors in numerical computations. Therefore, we estimate this by examining each row. Similar to the approach in \cite{Shifted}, we can express \eqref{eq:24} as
\begin{equation}
q_{i}^{\top}=x_{i}^{\top}(R+E_{Ri})^{-1}=x_{i}^{\top}(I+R^{-1}E_{Ri})^{-1}R^{-1}. \label{eq:217}
\end{equation}
When we define
\begin{equation}
(I+R^{-1}E_{i})^{-1}=I+ \theta_{i}, \label{eq:218}
\end{equation}
where
\begin{equation}
\theta_{i}:=\sum_{j=1}^{\infty}(-R^{-1}E_{Ri})^{j}, \label{eq:219}
\end{equation}
based on \eqref{eq:24} and \eqref{eq:25}, we can have
\begin{equation}
E_{Xi}^{\top}=x_{i}^{\top}\theta_{i}, \label{eq:220}
\end{equation}
which is the $i$-th row of $E_{X}$. Based on \eqref{eq:nn}, \eqref{eq:29}, \eqref{eq:223} and \eqref{eq:225}, we can bound $\norm{R^{-1}E_{Ri}}_{2}$ as
\begin{equation} \label{eq:r-1eri}
\begin{split}
\norm{R^{-1}E_{Ri}}_{2} &\le \norm{R^{-1}}_{2}\norm{E_{Ri}}_{2} \\ &\le \frac{1.03n\sqrt{n}\uu[X]_{g}}{\sqrt{(\sigma_{n}(X))^{2}+0.9s}} \\ &\le \frac{1.03n\sqrt{n}\uu[X]_{g}}{\sqrt{0.9s}} \\ &\le \frac{1.03n\sqrt{n}\uu[X]_{g}}{\sqrt{9.9(mn\uu+n(n+1)\uu)[X]_{g}^{2}}} \\ &\le \frac{1.03n\sqrt{n}\uu[X]_{g}}{\sqrt{9.9n(n+1)\uu[X]_{g}^{2}}} \\ &\le 0.35 \cdot \sqrt{n\uu} \\ &\le 0.1. 
\end{split}
\end{equation}
Putting \eqref{eq:223}, \eqref{eq:225}, \eqref{eq:r-1eri} into \eqref{eq:219} and we have
\begin{equation} \label{eq:theta}
\begin{split}
\norm{\theta_{i}}_{2} &\le \sum_{j=1}^{\infty}(\norm{R^{-1}}_{2}\norm{E_{Ri}}_{2})^{j} \\ &= \frac{\norm{R^{-1}}_{2}\norm{E_{Ri}}_{2}}{1-\norm{R^{-1}}_{2}\norm{E_{Ri}}_{2}} \\ &\le \frac{1}{0.9} \cdot \frac{1.03n\sqrt{n}\uu[X]_{g}}{\sqrt{(\sigma_{n}(X))^{2}+0.9s}} \\ &\le \frac{1.15n\sqrt{n}\uu[X]_{g}}{\sqrt{(\sigma_{n}(X))^{2}+0.9s}}. 
\end{split}
\end{equation}
Summing all the items of \eqref{eq:220} together and with \eqref{eq:theta}, we have
\begin{equation}
\norm{E_{X}}_{2} \le \norm{E_{X}}_{F} \le \norm{X}_{F}\norm{\theta_{i}}_{2} \le \frac{1.15n^{2}\uu[X]_{g}^{2}}{\sqrt{(\sigma_{n}(X))^{2}+0.9s}}, \nonumber
\end{equation}
when $\norm{X}_{F} \le \sqrt{n}[X]_{g}$. Therefore, Lemma~\ref{lemma 4.4} is proved.
\end{proof}
\begin{remark}
The derivations of Lemmas~\ref{lemma 4.1}-\ref{lemma 4.4} utilize the properties of $[\cdot]_{g}$ and we can get sharper upper bounds compared to those in \cite{Shifted}. This shows that Shifted CholeskyQR can be analyzed from the column of the input matrix $X$. The calculation of the Gram matrix and the existence of Cholesky factorization make it possible for us to improve the algorithm from this perspective.
\end{remark}

\subsection{Proof of Theorem~\ref{thm:41}}
\label{sec:45}
\begin{proof}
Using the previous lemmas in Section~\ref{sec:44}, we begin to estimate the orthogonality and residual of our improved Shifted CholeskyQR. The proof of Theorem~\ref{thm:41} is similar to that in \cite{Shifted}. We aim to demonstrate that comparable results hold, even with our enhanced bounds in the previous lemmas, based on the properties of $[\cdot]_{g}$ discussed in Section~\ref{sec:44}.

First, we consider the orthogonality. Based on \eqref{eq:25}, we can get
\begin{equation} \label{eq:QQ}
\begin{split}
Q^{\top}Q &= R^{-\top}(X+E_{X})^{\top}(X+E_{X})R^{-1} \\
&= R^{-\top}X^{\top}XR^{-1}+R^{-\top}X^{\top}E_{X}R^{-1} \\
&+ R^{-\top}E_{X}^{\top}XR^{-1}+R^{-\top}E_{X}^{\top}E_{X}R^{-1} \\
&= I-R^{-\top}(sI+E_{A}+E_{B})R^{-1}+(XR^{-1})^{\top}E_{X}R^{-1} \\
&+ R^{-\top}E_{X}^{\top}(XR^{-1})+R^{-\top}E_{X}^{\top}E_{X}R^{-1}. 
\end{split}
\end{equation}
With \eqref{eq:QQ}, we have
\begin{equation} \label{eq:229}
\begin{split}
\norm{Q^{\top}Q-I}_{2} &\le \norm{R^{-1}}_{2}^{2}(\norm{E_{A}}_{2}+\norm{E_{B}}_{2}+s)+2\norm{R^{-1}}_{2}\norm{XR^{-1}}_{2}\norm{E_{X}}_{2} \\ &+ \norm{R^{-1}}_{2}^{2}\norm{E_{X}}_{2}^{2}. 
\end{split}
\end{equation}
According to \eqref{eq:29}-\eqref{eq:eb}, we can get $\norm{E_{A}}_{2}+\norm{E_{B}}_{2} \le 1.1(mn\uu+n(n+1)\uu)[X]_{g}^{2} \le 0.1s$. With \eqref{eq:223}, we can get
\begin{equation} \label{eq:230}
\begin{split}
\norm{R^{-1}}_{2}^{2}(\norm{E_{A}}_{2}+\norm{E_{B}}_{2}+s) &\le \frac{1.1s}{(\sigma_{n}(X))^{2}+0.9s} \\ &\le \frac{11}{9} \\ &\le 1.23. 
\end{split}
\end{equation}
Based on \eqref{eq:29}, \eqref{eq:223}, \eqref{eq:224} and \eqref{eq:226}, we can obtain
\begin{equation} \label{eq:231}
\begin{split}
2\norm{R^{-1}}_{2}\norm{XR^{-1}}_{2}\norm{E_{X}}_{2} &\le 2 \cdot \frac{1}{\sqrt{(\sigma_{n}(X))^{2}+0.9s}} \cdot 1.5 \cdot \frac{1.15n^{2}\uu[X]_{g}^{2}}{\sqrt{(\sigma_{n}(X))^{2}+0.9s}} \\ &\le \frac{3.45n^{2}\uu[X]_{g}^{2}}{(\sigma_{n}(X))^{2}+0.9s} \\ &\le \frac{\frac{3.45}{11} \cdot s}{0.9s} \\ &\le 0.35. 
\end{split}
\end{equation}
With \eqref{eq:223} and \eqref{eq:226}, we have
\begin{equation} \label{eq:232}
\begin{split}
\norm{R^{-1}}_{2}^{2}\norm{E_{X}}_{2}^{2} &\le \frac{1}{(\sigma_{n}(X))^{2}+0.9s} \cdot \frac{(1.15n^{2}\uu[X]_{g}^{2})^{2}}{(\sigma_{n}(X))^{2}+0.9s} \\ &\le \frac{(\frac{3.45}{11} \cdot s)^{2}}{(0.9s)^{2}} \\ &\le 0.02. 
\end{split}
\end{equation}
We put \eqref{eq:230}-\eqref{eq:232} into \eqref{eq:229} and we can get
\begin{equation}
\begin{split}
\norm{Q^{\top}Q-I}_{2} &\le 1.23+0.35+0.02 \nonumber \\ &\le 1.6. \nonumber
\end{split}
\end{equation}
Therefore, \eqref{eq:227} is proved.

From \eqref{eq:227}, it is easy to see that
\begin{equation}
\norm{Q}_{2} \le 1.62. \label{eq:233}
\end{equation}
For the residual, from \eqref{eq:233}, we can easily get
\begin{equation}
\norm{Q}_{F} \le 1.62\sqrt{n}. \label{eq:234}
\end{equation}
For $\norm{QR-X}_{F}$, based on \eqref{eq:225} and \eqref{eq:234}, similar to the corresponding steps in \cite{Shifted}, we will have \eqref{eq:228}. In all, Theorem~\ref{thm:41} is proved
\end{proof}
\begin{remark}
In the proof of Theorem~\ref{thm:41}, we demonstrate that our improved $s$ is sufficient to ensure numerical stability for Shifted CholeskyQR, with enhanced bounds established in the previous lemmas. This represents significant progress compared to that in \cite{Shifted}. The residual in \eqref{eq:228} shows a tighter upper bound compared to that in \cite{Shifted}. More importantly, \eqref{eq:228} can improve the condition for $\kappa_{2}(X)$ in the estimation of the singular values of $Q$ in the next section.
\end{remark}

\subsection{Proof of Theorem~\ref{thm:42}}
\label{sec:46}
In this section, we give the proof for Theorem~\ref{thm:42}.

\begin{proof}
We have already estimated $\norm{Q}_{2}$. To estimate $\kappa_{2}(X)$, we need to estimate $\sigma_{n}(Q)$. The primary steps of analysis are similar to that in \cite{Shifted}. When \eqref{eq:25} holds, according to Lemma~\ref{lemma 2.1}, we can get
\begin{equation}
\sigma_{n}(Q) \ge \sigma_{n}(XR^{-1})-\norm{E_{X}R^{-1}}_{2}. \label{eq:237}
\end{equation}
With \eqref{eq:223} and \eqref{eq:226}, we can obtain
\begin{equation}
\norm{E_{X}R^{-1}}_{2} \le \norm{E_{X}}_{2}\norm{R^{-1}}_{2} \le \frac{1.67n^{2}\uu[X]_{g}}{(\sigma_{n}(X))^{2}+0.9s}. \label{eq:238}
\end{equation}
Using the similar method in \cite{Shifted}, we have
\begin{equation}
\sigma_{n}(XR^{-1}) \ge \frac{\sigma_{n}(X)}{\sqrt{(\sigma_{n}(X))^{2}+s}} \cdot 0.9. \label{eq:239}
\end{equation}
When \eqref{eq:28} holds, we put \eqref{eq:238} and \eqref{eq:239} into \eqref{eq:237} and with $t=\frac{s}{\norm{X}_{2}^{2}}$, we can get
\begin{equation} \label{eq:240}
\begin{split}
\sigma_{n}(Q) &\ge \frac{0.9\sigma_{n}(X)}{\sqrt{(\sigma_{n}(X))^{2}+s}}-\frac{1.67n^{2}\uu[X]_{g}}{\sqrt{(\sigma_{n}(X))^{2}+0.9s}} \\ &\ge \frac{0.9}{\sqrt{(\sigma_{n}(X))^{2}+s}} \cdot (\sigma_{n}(X)-\frac{1.67}{0.9 \cdot \sqrt{0.9}} \cdot n^{2}\uu[X]_{g}) \\ &\ge \frac{\sigma_{n}(X)}{2\sqrt{(\sigma_{n}(X))^{2}+s}} \\ &= \frac{1}{2\sqrt{1+t(\kappa_{2}(X))^{2}}}.
\end{split}
\end{equation}
Based on \eqref{eq:233} and \eqref{eq:240}, we have
\begin{equation}
\kappa_{2}(Q) \le 3.24 \cdot \sqrt{1+t(\kappa_{2}(X))^{2}}. \nonumber
\end{equation}
Therefore, we can get \eqref{eq:235}.

To improve the stability of orthogonality and residual, we add a CholeskyQR2 following the Shifted CholeskyQR, resulting in the Shifted CholeskyQR3. The numerical stability of this approach will be demonstrated in the next section similar to that in \cite{Shifted}. To obtain the sufficient condition of $\kappa_{2}(X)$ without encountering the numerical breakdown, based on \eqref{eq:delta} in \cite{error}, we let
\begin{equation}
\kappa_{2}(Q) \le 3.24\sqrt{1+t(\kappa_{2}(X))^{2}} \le \frac{1}{8\sqrt{mn\uu+n(n+1)\uu}}. \label{eq:241}
\end{equation}
When \eqref{eq:29} is satisfied, along with \eqref{eq:p} and $t=\frac{s}{\norm{X}_{2}^{2}}$, we can have $11p^{2}(mn\uu+n(n+1)\uu) \le t \le \frac{1}{100}p^{2}$. When $s=11(mn\uu+n(n+1)\uu)[X]_{g}^{2}$, $t=11p^{2}(mn\uu+n(n+1)\uu)$. If $\kappa_{2}(X)$ is large enough, \textit{e.g.}, $\kappa_{2}(X) \ge \uu^{-\frac{1}{2}}$, we can have $t(\kappa_{2}(X))^{2} \ge 11(m+n)>>1$. Therefore, $1+t(\kappa_{2}(X))^{2} \approx t(\kappa_{2}(X))^{2}$. With \eqref{eq:241}, we can conclude that
\begin{equation}
\kappa_{2}(X) \le \frac{1}{25.92\sqrt{t} \cdot \sqrt{mn\uu+n(n+1)\uu}}. \label{eq:2X}
\end{equation}
We put $t=11p^{2}(mn\uu+n(n+1)\uu)$ into \eqref{eq:2X} and we can obtain \eqref{eq:236}. Therefore, Theorem~\ref{thm:42} is proved.
\end{proof}
\begin{remark}
We have shown that our improved Shifted CholeskyQR, with a smaller $s$, has advantages in terms of the requirement for $\kappa_{2}(X)$ and its sufficient condition compared to the original method. A comprehensive comparison of the theoretical results is provided in Section~\ref{sec:introduction}, highlighting these advantages, which are further illustrated in Section~\ref{sec:experiments}.
\end{remark}

\subsection{Proof of Theorem~\ref{thm:43}}
\label{sec:47}
In this section, we prove Theorem~\ref{thm:43} with some results in Theorem~\ref{thm:41}.

\begin{proof}
We write CholeskyQR2 in Shifted CholeskyQR3 with error matrices below.
\begin{align}
C-Q^{\top}Q &= E_{1}, \nonumber \\
R_{1}^{\top}R_{1}-C &= E_{2}, \nonumber \\
Q_{1}R_{1}-Q &= E_{3}, \label{eq:246} \\
R_{1}R-R_{2} &= E_{4}, \label{eq:247} \\
C_{1}-Q_{1}^{\top}Q_{1} &= E_{5}, \nonumber \\
R_{3}^{\top}R_{3}-C_{1} &= E_{6}, \nonumber \\
Q_{2}R_{3}-Q_{1} &= E_{7}, \label{eq:250} \\
R_{3}R_{2}-R{4} &= E_{8}. \label{eq:251}
\end{align}

Similar to the proof of Theorem~\ref{thm:41}, we consider the orthogonality first. For our improved Shifted CholeskyQR3, similar to that in \cite{error}, when Shifted CholeskyQR3 is applicable, we can get
\begin{align}
\kappa_{2}(Q) &\le \frac{1}{8\sqrt{mn\uu+n(n+1)\uu}}, \label{eq:252} \\
\kappa_{2}(Q_{1}) &\le 1.1. \label{eq:253}
\end{align}
Therefore, we can obtain \eqref{eq:242}.

When considering the residual, based on \eqref{eq:246}-\eqref{eq:251}, we have
\begin{equation} \label{eq:Q2R4}
\begin{split}
Q_{2}R_{4} &= (Q_{1}+E_{7})R_{3}^{-1}(R_{3}R_{2}-E_{8}) \\ &= (Q_{1}+E_{7})R_{2}-(Q_{1}+E_{7})R_{3}^{-1}E_{8} \\ &= Q_{1}R_{2}+E_{7}R_{2}-Q_{2}E_{8} \\ &= (Q+E_{3})R_{1}^{-1}(R_{1}R-E_{4})+E_{7}R_{2}-Q_{2}E_{8} \\ &= (Q+E_{3})R-(Q+E_{3})R_{1}^{-1}E_{4}+E_{7}R_{2}-Q_{2}E_{8} \\ &= QR+E_{3}R-Q_{1}E_{4}+E_{7}R_{2}-Q_{2}E_{8}. 
\end{split}
\end{equation}
Therefore, with \eqref{eq:Q2R4}, it is obvious that
\begin{equation} \label{eq:254}
\begin{split}
\norm{Q_{2}R_{4}-X}_{F} &\le \norm{QR-X}_{F}+\norm{E_{3}}_{F}\norm{R}_{2}+\norm{Q_{1}}_{2}\norm{E_{4}}_{F} \\ &+ \norm{E_{7}}_{F}\norm{R_{2}}_{2}+\norm{Q_{2}}_{2}\norm{E_{8}}_{F}. 
\end{split}
\end{equation}
Similar to \eqref{eq:24}, we express \eqref{eq:246} in each row as $q_{1i}^{\top}=q_{i}^{\top}(R_{1}+E_{R1i})^{-1}$, where $q_{1i}^{\top}$ and $q_{i}^{\top}$ denote the $i$-th rows of $Q_{1}$ and $Q$. Following the methodologies outlined in \cite{Shifted, error} and the concepts presented in our work, we have
\begin{equation}
\norm{R}_{2} \le 1.006\norm{X}_{2}, \label{eq:255} 
\end{equation}
\begin{equation} \label{eq:257} 
\begin{split}
\norm{E_{R1i}}_{2} &\le 1.2n\sqrt{n}\uu \cdot \norm{Q}_{2} \\ 
&\le 2.079n\sqrt{n}\uu, 
\end{split}
\end{equation}
\begin{equation}
\norm{Q_{1}}_{2} \le 1.039, \label{eq:258} 
\end{equation}
\begin{equation} \label{eq:259}
\begin{split}
\norm{R_{1}}_{2} &\le 1.1\norm{Q}_{2} \\ &\le 1.906. 
\end{split}
\end{equation}
We combine \eqref{eq:255}-\eqref{eq:259} with Lemma~\ref{lemma 2.2}, Lemma~\ref{lemma 2.5}, \eqref{eq:rg} and similar steps in \cite{Shifted}, we can bound $\norm{E_{3}}_{F}$, $\norm{E_{4}}_{F}$ and $[E_{4}]_{g}$ in \eqref{eq:246} and \eqref{eq:247} as
\begin{equation} \label{eq:260} 
\begin{split}
\norm{E_{3}}_{F} &\le \norm{Q_{1}}_{F} \cdot \norm{E_{R1i}}_{2} \\ &\le 1.039 \cdot \sqrt{n} \cdot 2.079n\sqrt{n}\uu \\ &\le 2.16n^{2}\uu, 
\end{split}
\end{equation}
\begin{equation} \label{eq:261} 
\begin{split}
\norm{E_{4}}_{F} &\le \gamma_{n}(\norm{R_{1}}_{F} \cdot \norm{R}_{F}) \\ &\le \gamma_{n}(\sqrt{n} \cdot \norm{R_{1}}_{2} \cdot \sqrt{n} \cdot [R]_{g}) \\ &\le 1.1n^{2}\uu \cdot 1.906 \cdot 1.006p\norm{X}_{2} \\ &\le 2.11pn^{2}\uu\norm{X}_{2}, 
\end{split}
\end{equation}
\begin{equation} \label{eq:262}
\begin{split}
[E_{4}]_{g} &\le \gamma_{n}(\norm{R_{1}}_{F} \cdot [R]_{g}) \\ &\le \gamma_{n}(\sqrt{n}\norm{R_{1}}_{2} \cdot [R]_{g}) \\ &\le 1.1n\sqrt{n}\uu \cdot 1.906 \cdot 1.006p\norm{X}_{2} \\ &\le 2.11pn\sqrt{n}\uu\norm{X}_{2}. 
\end{split}
\end{equation}
Moreover, based on Lemma~\ref{lemma 2.5}, Lemma~\ref{lemma 2.6}, \eqref{eq:rg}, \eqref{eq:255}, \eqref{eq:261} and \eqref{eq:262}, $\norm{R_{2}}_{2}$ and $[R_{2}]_{g}$ in \eqref{eq:247} can be bounded as
\begin{equation} \label{eq:263} 
\begin{split}
\norm{R_{2}}_{2} &\le \norm{R_{1}}_{2}\norm{R}_{2}+\norm{E_{4}}_{2} \\ &\le 1.906 \cdot 1.006\norm{X}_{2}+2.11pn^{2}\uu\norm{X}_{2} \\ &\le 1.95\norm{X}_{2}, 
\end{split}
\end{equation}
\begin{equation} \label{eq:264}
\begin{split}
[R_{2}]_{g} &\le \norm{R_{1}}_{2}[R]_{g}+[E_{4}]_{g} \\ &\le 1.906 \cdot 1.006p\norm{X}_{2}+2.11pn\sqrt{n}\uu\norm{X}_{2} \\ &\le 1.95p\norm{X}_{2}. 
\end{split}
\end{equation}
Similar to \eqref{eq:24}, we write \eqref{eq:250} in each row as $q_{2i}^{\top}=q_{1i}^{\top}(R_{3}+E_{R3i})^{-1}$, where $q_{2i}^{\top}$ and $q_{1i}^{\top}$ represent the $i$-th rows of $Q_{2}$ and $Q_{1}$. Similar to \eqref{eq:257}-\eqref{eq:259}, with \eqref{eq:mn}, \eqref{eq:nn} and \eqref{eq:242}, we can get
\begin{equation}
\norm{Q_{2}}_{2} \le 1.1, \label{eq:265} 
\end{equation}
\begin{equation} \label{eq:266} 
\begin{split}
\norm{E_{R3i}}_{2} &\le 1.2n\sqrt{n}\norm{Q_{1}}_{2} \\ &\le 1.2n\sqrt{n}\uu \cdot 1.039 \\ &\le 1.246n\sqrt{n}\uu, 
\end{split}
\end{equation}
\begin{equation} \label{eq:267}
\begin{split}
\norm{R_{3}}_{2} &\le 1.1\norm{Q_{1}}_{2} \\ &\le 1.143, 
\end{split}
\end{equation}
With Lemma~\ref{lemma 2.2} and \eqref{eq:264}-\eqref{eq:267}, we can bound $\norm{E_{7}}_{F}$ and $\norm{E_{8}}_{F}$ in \eqref{eq:250} and \eqref{eq:251} as
\begin{equation} \label{eq:268} 
\begin{split}
\norm{E_{7}}_{F} &\le \norm{Q_{2}}_{F} \cdot \norm{E_{R3i}}_{2} \\ &\le 1.1\sqrt{n} \cdot 1.246n\sqrt{n}\uu \\ &\le 1.38n^{2}\uu, 
\end{split}
\end{equation}
\begin{equation} \label{eq:269}
\begin{split}
\norm{E_{8}}_{F} &\le \gamma_{n}(\norm{R_{3}}_{F} \cdot \norm{R_{2}}_{F}) \\ &\le \gamma_{n}(\sqrt{n} \cdot \norm{R_{3}}_{2} \cdot \sqrt{n} \cdot [R_{2}]_{g}) \\ &\le 1.1pn^{2}\uu \cdot 1.143 \cdot 1.95p\norm{X}_{2} \\ &\le 2.46pn^{2}\uu\norm{X}_{2}. 
\end{split}
\end{equation}
Therefore, we put \eqref{eq:228}, \eqref{eq:255}, \eqref{eq:258}, \eqref{eq:260}, \eqref{eq:261}, \eqref{eq:263}, \eqref{eq:265}, \eqref{eq:268} and \eqref{eq:269} into \eqref{eq:254} and we can get \eqref{eq:243}. In all, Theorem~\ref{thm:43} is proved.
\end{proof}
\begin{remark}
Based on \eqref{eq:243}, we find that we obtain a sharper upper bound of the residual of the algorithm compared to that in \cite{Shifted}, utilizing the properties of $[\cdot]_{g}$. This represents a theoretical advancement in rounding error analysis. The steps leading to \eqref{eq:264} highlight the effectiveness of Lemma~\ref{lemma 2.5} and Lemma~\ref{lemma 2.7}. Although the second inequality of \eqref{eq:A3} appears weaker than the first inequality of \eqref{eq:A3}, it cannot be dismissed in estimating $[\cdot]_{g}$ of the error matrix in terms of its absolute value. This lays a solid foundation for \eqref{eq:264} and \eqref{eq:269}, marking advancements in estimation methods for problems related to matrix multiplications.
\end{remark}

Moreover, if $X$ is not highly ill-conditioned, meaning that $\kappa_{2}(X)$ is small, our estimation of $\norm{E_{A}}_{2}$ and $\norm{E_{B}}_{2}$ can also be directly applied to CholeskyQR2. Therefore, the sufficient condition for $\kappa_{2}(X)$ can be expressed as
\begin{equation}
\kappa_{2}(X) \le \frac{1}{8p\sqrt{mn\uu+n(n+1)\uu}}. \nonumber
\end{equation}
This condition is a better sufficient condition compared to \eqref{eq:delta} in \cite{error}.

\section{Experimental Results}
\label{sec:experiments}
In this study, we conduct numerical experiments using MATLAB R2022a on a laptop. We compare our improved Shifted CholeskyQR3 with the original Shifted CholeskyQR3, focusing on three key properties: numerical stability(assessed through orthogonality $\norm{{Q_{2}}^{\top}{Q_{2}}-I}_{F}$ and residual $\norm{Q_{2}R_{4}-X}_{F}$ for Shifted CholeskyQR), the condition number of $Q$(denoted as $\kappa_{2}(Q)$) and the computational time(CPU time measured in seconds). Additionally, we present the $p$-value, defined as $p=\frac{[X]_{g}}{\norm{X}_{2}}$, to illustrate the extent of improvement brought by our reduced $s$ compared to the original method in \cite{Shifted}. As a comparison group, we also evaluate the properties of HouseholderQR, which is considered one of the most stable numerical algorithms, to demonstrate the effectiveness and advantages of our improved Shifted CholeskyQR3. The specifications of our computer used for these experiments are provided in Table~\ref{tab:C}. We assess the performance of our method in multi-core CPU environments. 

\begin{table}[H]
\begin{center}
\caption{The specifications of our computer}
\centering
\begin{tabular}{c|c}
\hline
Item & Specification\\
\hline \hline
System & Windows 11 family(10.0, Version 22000) \\
BIOS & GBCN17WW \\
CPU & Intel(R) Core(TM) i5-10500H CPU @ 2.50GHz  -2.5 GHz \\
Number of CPUs / node & 12 \\
Memory size / node & 8 GB \\
Direct Version & DirectX 12 \\
\hline
\end{tabular}
\label{tab:C}
\end{center}
\end{table}

\subsection{Numerical examples}
In this part, we introduce the numerical examples, specifically the test matrix $X$ utilized in this work. The primary test matrix $X \in \mathbb{R}^{m\times n}$ is similar to that used in \cite{Shifted, error} and is constructed by SVD. It is straightforward to observe the influence of $\kappa_{2}(X)$, $m$ and $n$ while controlling the other two factors. Additionally, to test the applicability and the numerical stability of our improved Shifted CholeskyQR3, we present two examples widely used in engineering and other fields.

\subsubsection{The input $X$ based on SVD}
We first construct the matrix $X$ for the numerical experiments using Singular Value Decomposition (SVD), similar to the approach described in \cite{Shifted, error}. We control $\kappa_{2}(X)$ through $\sigma_{n}(X)$. Specifically, we set
\begin{equation}
X=U \Sigma V^{T}. \nonumber
\end{equation}
Here, $U \in \mathbb{R}^{m\times m}, V \in \mathbb{R}^{n\times n}$ are random orthogonal matrices and
\begin{equation}
\Sigma = {\rm diag}(1, \sigma^{\frac{1}{n-1}}, \cdots, \sigma^{\frac{n-2}{n-1}}, \sigma) \in \mathbb{R}^{m\times n}. \nonumber
\end{equation}
Here, $0<\sigma<1$ is a constant. Therefore, we have $\sigma_{1}(X)=\norm{X}_{2}=1$ and $\kappa_{2}(X)=\frac{1}{\sigma}$.

\subsubsection{The Hilbert matrix}
The Hilbert matrix $X \in \mathbb{R}^{n\times n}$ is a well-known ill-conditioned matrix. It is widely used in many applications, including numerical approximation theory and solving linear systems, see \cite{Ein, Beckermann, Tricks} and the references therein. As $n$ increases, $\kappa_{2}(X)$ also increases. The Hilbert matrix $X$ is defined as 
\begin{equation}
X_{ij} = \frac{1}{i+j-1}, i,j=1,2,\cdots,n. \nonumber
\end{equation}
We can use $X=hilb(n)$ in MATLAB to receive a Hilbert matrix $X \in \mathbb{R}^{n\times n}$. 

\subsubsection{The arrowhead matrix}
The arrowhead matrix $X \in \mathbb{R}^{n\times n}$ plays an important role in graph theory, control theory and some eigenvalue problems, see \cite{Constructing, Li, Eigen, Accurate, JCP} and the references therein. Its primary characteristic is that all the elements are zero except for those in the first column, the first row and the diagonal. In this work, we define an arrowhead matrix as follows.
\begin{align}
X_{1j} &= 30, j=1,2,\cdots,n, \nonumber \\
X_{ii}&= 10, i=2,3,\cdots,n-1, \nonumber \\
X_{ii} &= 10^{-16}, i=n, \nonumber \\
X_{ij} &= 0, \mbox{others} \nonumber.
\end{align}

\subsection{Numerical stability of the algorithms}
In this section, we test the numerical stability of the algorithms. To assess this, we conduct experiments considering three factors: $\kappa_{2}(X)$, $m$ and $n$ to demonstrate the properties of Shifted CholeskyQR3. For clarity, we refer to our improved Shifted CholeskyQR3 as 'Improved', while the original Shifted CholeskyQR3 is referred to as 'Original'.

To assess the potential influence of $\kappa_{2}(X)$, we obtain $X$ using SVD first. We fix $m=2048$ and $n=64$, varying $\sigma$ to evaluate the effectiveness of our algorithm with different $\kappa_{2}(X)$. The numerical results are listed in Table~\ref{tab:Ok} and Table~\ref{tab:Rk}. Numerical experiments show that our improved Shifted CholeskyQR3 exhibit better orthogonality and residual compared to HouseholderQR, demonstrating strong numerical stability. The numerical stability of our improved algorithm is comparable to that of the original Shifted CholeskyQR3. A key advantage of our improved Shifted CholeskyQR3 over the original one is that our improved algorithm can handle more ill-conditioned $X$ with $\kappa_{2}(X) \ge 10^{12}$. The conservative choice of $s$ in the original Shifted CholeskyQR3 limits its computational range, as reflected in the comparison of $\kappa_{2}(X)$ between \eqref{eq:c2} and \eqref{eq:236}. In our practical example of the Hilbert matrix, we take $n=12$ and $\kappa_{2}(X)=1.62e+16$. In the example of the arrowhead matrix, we take $n=64$ and $\kappa_{2}(X)=3.40e+18$. The numerical results are shown in Table~\ref{tab:Hilbert} and Table~\ref{tab:arrowhead}. They also demonstrate that our improved Shifted CholeskyQR3 has better applicability and is able to handle more ill-conditioned matrices effectively than the original one. 

To examine the influence of $m$ and $n$, we construct $X$ based on SVD while maintaining $\kappa_{2}(X)=10^{12}$. When $m$ is varying, we keep $n=64$. When $n$ is varying, we keep $m=2048$. The numerical results are presented in Table~\ref{tab:Om}-~\ref{tab:Rn}. Our findings indicate that the increasing $n$ leads to greater rounding errors in orthogonality and residual, while $m$ does not impact these aspects significantly. Our improved Shifted CholeskyQR3 maintains a level of the numerical stability comparable to that of the original Shifted CholeskyQR3 and is more accurate compared to HouseholderQR across various values of $m$ and $n$. This set of experiments shows that our improved Shifted CholeskyQR3 is numerical stable across different problem sizes.

Overall, our examples demonstrate that our improved Shifted CholeskyQR3 is more applicable for ill-conditioned matrices without sacrificing numerical stability, performing at a level comparable to the original Shifted CholeskyQR3. In many cases, it even exhibits better accuracy compared to the traditional HouseholderQR.

\begin{table}
\caption{Orthogonality of the algorithms with $\kappa_{2}(X)$ varying when $m=2048$ and $n=64$}
\centering
\begin{tabular}{||c c c c c c||}
\hline
$\kappa_{2}(X)$ & $1.00e+8$ & $1.00e+10$ & $1.00e+12$ & $1.00e+14$ & $1.00e+16$ \\
\hline
Improved & $2.07e-15$ & $2.04e-15$ & $2.03e-15$ & $2.04e-15$ & - \\
\hline
Original & $2.14e-15$ & $2.21e-15$ & $1.90e-15$ & - & - \\
\hline
HouseholderQR & $2.77e-15$ & $2.46e-15$ & $2.48e-15$ & $2.75e-14$ & $2.67e-15$ \\
\hline
\end{tabular}
\label{tab:Ok}
\end{table}

\begin{table}
\caption{Residual of the algorithms with $\kappa_{2}(X)$ varying when $m=2048$ and $n=64$}
\centering
\begin{tabular}{||c c c c c c||}
\hline
$\kappa_{2}(X)$ & $1.00e+8$ & $1.00e+10$ & $1.00e+12$ & $1.00e+14$ & $1.00e+16$ \\
\hline
Improved & $6.35e-16$ & $6.01e-16$ & $5.80e-16$ & $5.64e-16$ & - \\
\hline
Original & $6.67e-16$ & $6.20e-16$ & $6.22e-16$ & - & - \\
\hline
HouseholderQR & $1.26e-15$ & $1.38e-15$ & $1.27e-15$ & $1.27e-15$ & $9.61e-16$ \\
\hline
\end{tabular}
\label{tab:Rk}
\end{table}

\begin{table}
\caption{Numerical results for the Hilbert matrix with $n=12$}
\centering
\begin{tabular}{||c c c|}
\hline
Algorithm & Improved & Original \\
\hline
Orthogonality & $3.59e-15$ & $-$ \\
\hline
Residual & $2.14e-16$ & $-$ \\
\hline
\end{tabular}
\label{tab:Hilbert}
\end{table}

\begin{table}
\caption{Numerical results for the arrowhead matrix with $n=64$}
\centering
\begin{tabular}{||c c c|}
\hline
Algorithm & Improved & Original \\
\hline
Orthogonality & $1.24e-14$ & $-$ \\
\hline
Residual & $1.40e-14$ & $-$ \\
\hline
\end{tabular}
\label{tab:arrowhead}
\end{table}

\begin{table}
\caption{Orthogonality of all the algorithms with $m$ varying when $\kappa_{2}(X)=10^{12}$ and $n=64$}
\centering
\begin{tabular}{||c c c c c c||}
\hline
$m$ & $128$ & $256$ & $512$ & $1024$ & $2048$ \\
\hline
Improved & $3.62e-15$ & $4.07e-15$ & $3.11e-15$ & $2.12e-15$ & $2.03e-15$ \\
\hline
Original & $3.31e-15$ & $3.93e-15$ & $2.89e-15$ & $2.36e-15$ & $1.90e-15$ \\
\hline
HouseholderQR & $6.54e-15$ & $6.35e-15$ & $3.56e-15$ & $2.80e-15$ & $2.48e-15$ \\
\hline
\end{tabular}
\label{tab:Om}
\end{table}

\begin{table}
\caption{Residual of all the algorithms with $m$ varying when $\kappa_{2}(X)=10^{12}$ and $n=64$}
\centering
\begin{tabular}{||c c c c c c||}
\hline
$m$ & $128$ & $256$ & $512$ & $1024$ & $2048$ \\
\hline
Improved & $6.04e-16$ & $5.92e-16$ & $6.08e-16$ & $6.06e-16$ & $5.80e-16$ \\
\hline
Original & $6.09e-16$ & $5.91e-16$ & $5.95e-16$ & $5.86e-16$ & $6.22e-16$ \\
\hline
HouseholderQR & $7.31e-16$ & $9.45e-16$ & $7.55e-16$ & $7.48e-16$ & $1.27e-15$ \\
\hline
\end{tabular}
\label{tab:Rm}
\end{table}

\begin{table}
\caption{Orthogonality of all the algorithms with $n$ varying when $\kappa_{2}(X)=10^{12}$ and $m=2048$}
\centering
\begin{tabular}{||c c c c c c||}
\hline
$n$ & $64$ & $128$ & $256$ & $512$ & $1024$ \\
\hline
Improved & $2.03e-15$ & $3.25e-15$ & $5.29e-15$ & $9.53e-15$ & $1.69e-14$ \\
\hline
Original & $1.90e-15$ & $3.33e-15$ & $5.19e-15$ & $1.66e-15$ & $1.77e-14$ \\
\hline
HouseholderQR & $2.48e-15$ & $4.66e-15$ & $9.39e-15$ & $2.07e-14$ & $5.02e-14$ \\
\hline
\end{tabular}
\label{tab:On}
\end{table}

\begin{table}
\caption{Residual of all the algorithms with $n$ varying when $\kappa_{2}(X)=10^{12}$ and $m=2048$}
\centering
\begin{tabular}{||c c c c c c||}
\hline
$n$ & $64$ & $128$ & $256$ & $512$ & $1024$ \\
\hline
Improved & $5.80e-16$ & $1.07e-15$ & $2.01e-15$ & $3.06e-15$ & $4.32e-15$ \\
\hline
Original & $6.22e-16$ & $1.08e-15$ & $2.04e-15$ & $3.08e-15$ & $4.33e-15$ \\
\hline
HouseholderQR & $1.27e-15$ & $1.76e-15$ & $2.55e-15$ & $3.62e-15$ & $5.00e-15$ \\
\hline
\end{tabular}
\label{tab:Rn}
\end{table}

\subsection{$\kappa_{2}(Q)$ under different conditions}
In this group of experiments, we evaluate the impact of $\kappa_{2}(X)$, $m$ and $n$ on $\kappa_{2}(Q)$ using different values of $s$ for Shifted CholeskyQR3, which is crucial for assessing the applicability of the algorithms. We compare our improved Shifted CholeskyQR3 with the original Shifted CholeskyQR3.

In this group of experiments, we use $X$ based on SVD. Initially, we fix $m=2048$ and $n=64$, varying $\kappa_{2}(X)$ to see the corresponding $\kappa_{2}(Q)$ with different values of $s$ in Shifted CholeskyQR3. The results are listed in Table~\ref{tab:QK}. From Table~\ref{tab:QK}, we can see that $\kappa_{2}(X)$ exhibits a nearly direct proportionality to $\kappa_{2}(Q)$. With an improved smaller $s$, our improved Shifted CholeskyQR3 achieves a smaller $\kappa_{2}(X)$ compared to the original Shifted CholeskyQR3, which is consistent with \eqref{eq:qx} and \eqref{eq:235}.

Next, we test the influence of $m$ and $n$ on $\kappa_{2}(X)$. When varying $m$, we fix $\kappa_{2}(X)=10^{12}$ and $n=64$. For different $n$, we set $\kappa_{2}(X)=10^{12}$ and $m=2048$. The numerical results are listed in Table~\ref{tab:Qm} and Table~\ref{tab:Qn}. These results indicate that when dealing with a tall-skinny matrix $X \in \mathbb{R}^{m\times n}$ with $m>n$, increasing both $m$ and $n$ leads to a larger $\kappa_{2}(Q)$ while keeping $\kappa_{2}(X)$ fixed. This arises from the structures of both our improved $s$ and the original $s$. Across Table~\ref{tab:QK}-~\ref{tab:Qn}, we consistently observe that our method achieves a smaller $\kappa_{2}(Q)$ compared to the original Shifted CholeskyQR3, demonstrating the effectiveness of the improved $s$.

In conclusion, our reduced $s$ in this work results in a smaller $\kappa_{2}(Q)$, enhancing the applicability of our improved Shifted CholesyQR3 compared to the original algorithm. This represents a significant advancement in our research.

\begin{table}
\caption{$\kappa_{2}(Q)$ with $\kappa_{2}(X)$ varying with different $s$ when $m=2048$ and $n=64$}
\centering
\begin{tabular}{||c c c c c c||}
\hline
$\kappa_{2}(X)$ & $1.00e+8$ & $1.00e+10$ & $1.00e+12$ & $1.00e+14$ & $1.00e+16$ \\
\hline
Improved & $358.60$ & $3.37e+04$ & $3.18e+06$ & $3.01e+08$ & - \\
\hline
Original & $1.29e+03$ & $1.29e+05$ & $1.29e+07$ & - & - \\
\hline
\end{tabular}
\label{tab:QK}
\end{table}

\begin{table}
\caption{$\kappa_{2}(Q)$ with $m$ varying using different $s$ when $\kappa_{2}(X)=10^{12}$ and $n=64$}
\centering
\begin{tabular}{||c c c c c c||}
\hline
$m$ & $128$ & $256$ & $512$ & $1024$ & $2048$ \\
\hline
Improved & $9.62e+05$ & $1.24e+06$ & $1.66e+06$ & $2.29e+06$ & $3.18e+06$ \\
\hline
Original & $3.88e+06$ & $5.01e+06$ & $6.72e+06$ & $9.23e+06$ & $1.29e+07$ \\
\hline
\end{tabular}
\label{tab:Qm}
\end{table}

\begin{table}
\caption{$\kappa_{2}(Q)$ with $n$ varying using different $s$ when $\kappa_{2}(X)=10^{12}$ and $m=2048$}
\centering
\begin{tabular}{||c c c c c c||}
\hline
$n$ & $64$ & $128$ & $256$ & $512$ & $1024$ \\
\hline
Improved & $3.18e+06$ & $4.24e+06$ & $5.76e+06$ & $8.11e+06$ & $1.11e+07$ \\
\hline
Original & $1.29e+07$ & $1.84e+07$ & $2.68e+07$ & $4.00e+07$ & $6.20e+07$ \\
\hline
\end{tabular}
\label{tab:Qn}
\end{table}

\subsection{CPU times of the algorithms}
In addition to considering numerical stability and $\kappa_{2}(Q)$, we also need to take into account the CPU time required by these algorithms to demonstrate the efficiency of our improved algorithm. We test the corresponding CPU time with respect to the two variables, $m$ and $n$.

Similar to the previous section, we use $X$ based on SVD. For varying values of $m$, we set $n=64$ and $\kappa_{2}(X)=10^{12}$. When $n$ is varying, we fix $m=2048$ and $\kappa_{2}(X)=10^{12}$. We observe the variation in CPU time for our improved Shifted CholeskyQR3, the original Shifted CholeskyQR3 algorithm and HouseholderQR. The CPU times for these algorithms are listed in Table~\ref{tab:t1m} and Table~\ref{tab:t1n}. Numerical experiments show that both our improved Shifted CholeskyQR3 and the original Shifted CholeskyQR3 are significantly more efficient compared to HouseholderQR, highlighting a primary drawback of the widely-used HouseholderQR. Our improved Shifted CholeskyQR3 exhibits comparable speed to the original Shifted CholeskyQR3 with $normest$. Additionally, $n$ has a greater influence on CPU time compared to $m$. However, as both $m$ and $n$ increase, our improved Shifted CholeskyQR3 maintains a level of efficiency similar to that of the original Shifted CholeskyQR3. Therefore, we conclude that our improved Shifted CholeskyQR3 is an efficient algorithm with good accuracy for problems with moderate sizes.

\begin{table}
\caption{CPU time with $m$ varying (in second) when $\kappa_{2}(X)=10^{12}$ and $n=64$}
\centering
\begin{tabular}{||c c c c c c||}
\hline
$m$ & $128$ & $256$ & $512$ & $1024$ & $2048$ \\
\hline
Improved & $6.90e-04$ & $8.65e-04$ & $1.70e-03$ & $3.80e-03$ & $4.70e-03$ \\
\hline
Original & $2.10e-03$ & $9.55e-04$ & $1.50e-03$ & $4.40e-03$ & $6.20e-03$ \\
\hline
HouseholderQR & $1.21e-02$ & $3.45e-02$ & $3.38e-01$ & $2.00e+00$ & $1.24e+01$ \\
\hline
\end{tabular}
\label{tab:t1m}
\end{table}

\begin{table}
\caption{CPU time with $n$ varying (in second) when $\kappa_{2}(X)=10^{12}$ and $m=2048$}
\centering
\begin{tabular}{||c c c c c c||}
\hline
$n$ & $64$ & $128$ & $256$ & $512$ & $1024$ \\
\hline
Improved & $4.70e-03$ & $1.25e-02$ & $4.66e-02$ & $9.80e-02$ & $3.52e-01$ \\
\hline
Original & $6.20e-03$ & $1.46e-02$ & $4.59e-02$ & $9.02e-02$ & $4.45e-01$ \\
\hline
HouseholderQR & $1.12e+01$ & $2.59e+01$ & $5.66e+01$ & $1.16e+02$ & $3.11e+02$ \\
\hline
\end{tabular}
\label{tab:t1n}
\end{table}

\subsection{$p$-values}
Here, we aim to show the $p$-values in this work by using some examples. Based on Table~\ref{tab:Comparison} and \eqref{eq:p}, we can find that the proportion of our improved $s$ to the original $s$ is $p^{2}$. Therefore, the $p$-value reflects how much the shifted item $s$ is reduced according to our definition of $[\cdot]_{g}$. In the future, we will investigate how to estimate $p$ under different cases.

In this part, we test the $p$-value with varying values of $m$ and $n$ using $X$ based on SVD. With $m$ varying, we fix $n=64$ and $\kappa_{2}(X)=10^{12}$. For different values of $n$, we fix $m=2048$ and $\kappa_{2}(X)=10^{12}$. The numerical experiments are listed in Table~\ref{tab:t2m} and Table~\ref{tab:t2n}. The numerical results indicate that $p$ is relatively small compared to $1$. Notably, $n$ significantly influences $p$ more than $m$. With $n$ increasing, $p$ decreases markedly, which aligns with the theoretical lower bound of the $p$-value. This observation suggests that our improved $s$ is likely more effective for relatively large matrices.

\begin{table}
\caption{$p$ with $m$ varying when $\kappa_{2}(X)=10^{12}$ and $n=64$}
\centering
\begin{tabular}{||c c c c c c||}
\hline
$m$ & $128$ & $256$ & $512$ & $1024$ & $2048$ \\
\hline
$p$ & $0.2824$ & $0.2762$ & $0.2386$ & $0.2453$ & $0.2498$ \\
\hline
\end{tabular}
\label{tab:t2m}
\end{table}

\begin{table}
\caption{$p$ with $n$ varying when $\kappa_{2}(X)=10^{12}$ and $m=2048$}
\centering
\begin{tabular}{||c c c c c c||}
\hline
$n$ & $64$ & $128$ & $256$ & $512$ & $1024$ \\
\hline
$p$ & $0.2498$ & $0.2396$ & $0.2127$ & $0.2024$ & $0.1726$ \\
\hline
\end{tabular}
\label{tab:t2n}
\end{table}

\section{Discussions and Conclusions}
\label{sec:discussions}
This study focuses on determining an optimal choice for the shifted item $s$ based on the properties of the input matrix $X$ for Shifted CholeskyQR3. We introduce a new $[X]_{g}$ for $X$ based on column properties and derive a new smaller $s$ using $[X]_{g}$. We demonstrate that this smaller $s$ provides a better upper bound of $\kappa_{2}(X)$, thereby enhancing the applicability of Shifted CholeskyQR3 while maintaining its numerical stability in terms of both orthogonality and residuals. In terms of computational efficiency, our improved Shifted CholeskyQR3 outperforms the commonly used HouseholderQR method and exhibits a similar CPU time to the original Shifted CholeskyQR3 for moderately sized matrices, demonstrating that our algorithm is effective in terms of speed as a three-step deterministic method.

There are still several issues that need to be addressed in the future. Specifically, $[\cdot]_{g}$ of the matrix warrants further exploration. Developing efficient methods to quickly estimate $[X]_{g}$ for input matrices $X$ remains an open topic for future research, particularly for large matrices. In this work, we calculate $[X]_{g}$ by comparing the $2$-norms of the columns of $X$. However, as the size of the matrix increases, the CPU time and computational cost of this method increase significantly. Therefore, new techniques for estimating $[X]_{g}$ more effectively need to be developed. Moreover, the process of calculating $[X]_{g}$ indicates that parallel computing can be employed to obtain $[\cdot]_{g}$ more efficiently. We are currently developing an estimator that leverages parallel computing to calculate $[\cdot]_{g}$ for the improved Shifted CholeskyQR, which are useful for the improved Shifted CholeskyQR3 on GPU using parallel computing. In this study, we leverage the connections between $[\cdot]_{g}$ and other norms to conduct rounding error analysis. Given that $[\cdot]_{g}$ can be applied to various problems, such as HouseholderQR and Nyström approximation, we aim to explore its relationship with the singular values of the matrix and other factors, such as the condition number. We are also focusing on more properties related to $[\cdot]_{g}$.

Recent years have seen significant advancements in CholeskyQR methodologies, as evidenced by studies such as \cite{LUChol, 2016, Mixed}. While deterministic methods like Shifted CholeskyQR3 offer good accuracy, they are often relatively slow. The choice of the parameter $s$ continues to influence the applicability of the algorithm, even with the improvements proposed in this work. In addition, randomized methods for the CholeskyQR algorithm have been introduced \cite{Randomized, Novel} in recent years. However, all Cholesky-type algorithms face with issues of the sufficient condition for the condition number $\kappa_{2}(X)$ applicable to the input matrix $X$, which limits their practical use in industry. To address this, we are exploring new preconditioning steps designed to balance speed, accuracy, and applicability, thereby enhancing the performance of CholeskyQR-type algorithms.

\section*{Acknowledgments}
This work is supported by the CAS AMSS-PolyU Joint Laboratory of Applied Mathematics. The contributions of H. Guan and Z. Qiao are funded by the Hong Kong Research Grants Council through the RFS grant RFS2021-5S03 and GRF grant 15302122, as well as by the Hong Kong Polytechnic University under grant 4-ZZLS. We would like to express our gratitude to Mr. Yuan Liang from Beijing Normal University, Zhuhai, for his valuable suggestions regarding the coding aspects of this research. Additionally, we appreciate the insightful discussions with Mr. Renfeng Peng from the Chinese Academy of Sciences, Professor Valeria Simoncini and Dr. Davide Palitta from University of Bologna regarding the properties of $[\cdot]_{g}$ and potential future directions in this area. Our thanks also go to Dr. Nan Zheng from the Hong Kong Polytechnic University for her assistance in revising this manuscript. Finally, we are grateful to the two anonymous referees for their constructive feedback, which has contributed to enhancing this work. 

\section*{Conflict of interest}
The authors declare that they have no conflict of interest.

\section*{Data availability}
The authors declare that all data supporting the findings of this study are available within this article.

\bibliographystyle{plain}

\bibliography{references}

\begin{thebibliography}{10}

\bibitem{Randomized}
O.~Balabanov.
\newblock {Randomized CholeskyQR factorizations}.
\newblock {\em {arxiv preprint arXiv:2210.09953}}, 2022.

\bibitem{ballard2011}
G.~Ballard, J.~Demmel, O.~Holtz, and O.~Schwartz.
\newblock {Minimizing communication in numerical linear algebra}.
\newblock {\em {SIAM Journal on Matrix Analysis and Applications}}, 32(3):866--901, 2011.

\bibitem{Beckermann}
B.~Beckermann.
\newblock {The condition number of real Vandermonde, Krylov and positive definite Hankel matrices}.
\newblock {\em {Numerische Mathematik}}, 85:553--577, 2000.

\bibitem{Constructing}
A.~Borobia.
\newblock {Constructing matrices with prescribed main-diagonal submatrix and characteristic polynomial}.
\newblock {\em {Linear Algebra and its Applications}}, 418:886--890, 2006.

\bibitem{Tricks}
M.~Choi.
\newblock {Tricks or Treats with the Hilbert Matrix}.
\newblock {\em {The American Mathematical Monthly}}, 90(5):301--312, 1983.

\bibitem{2011}
P.G. Constantine and D.F. Gleich.
\newblock {Tall and skinny QR factorizations in MapReduce architectures}.
\newblock In {\em {Proceedings of the second international workshop on MapReduce and its applications}}, pages 43--50, 2011.

\bibitem{Floating}
J.~Demmel.
\newblock {On floating point errors in Cholesky}.
\newblock {\em {Tech.Report 14, LAPACK working Note}}, 1989.

\bibitem{Communication}
J.~Demmel, L.~Grigori, M.~Hoemmen, and J.~Langou.
\newblock {Communication-optimal parallel and sequential QR and LU factorizations}.
\newblock {\em {in Proceedings of the Conference on High Performance Computing Networking, Storage and Analysis}}, pages 36:1--36:12, 2009.

\bibitem{2018robust}
J.A. Duersch, M.~Shao, C.~Yang, and M.~Gu.
\newblock {A robust and efficient implementation of LOBPCG}.
\newblock {\em {SIAM Journal on Scientific Computing}}, 40(5):{C655--C676}, 2018.

\bibitem{Novel}
Y.~Fan, Y.~Guo, and T.~Lin.
\newblock {A Novel Randomized XR-Based Preconditioned CholeskyQR Algorithm}.
\newblock {\em {arxiv preprint arXiv:2111.11148}}, 2021.

\bibitem{Shifted}
T.~Fukaya, R.~Kannan, Y.~Nakatsukasa, Y.~Yamamoto, and Y.~Yanagisawa.
\newblock {Shifted Cholesky QR for computing the QR factorization of ill-conditioned matrices}.
\newblock {\em {SIAM Journal on Scientific Computing}}, 42(1):{A477--A503}, 2020.

\bibitem{2014}
T.~Fukaya, Y.~Nakatsukasa, Y.~Yanagisawa, and Y.~Yamamoto.
\newblock {CholeskyQR2: a simple and communication-avoiding algorithm for computing a tall-skinny QR factorization on a large-scale parallel system}.
\newblock In {\em {2014 5th workshop on latest advances in scalable algorithms for large-scale systems}}, pages 31--38. {IEEE}, 2014.

\bibitem{MatrixC}
G.H. Golub and C.F. Van~Loan.
\newblock {\em {Matrix Computations}}.
\newblock {The Johns Hopkins University Press}, {Baltimore}, {4th} edition, 2013.

\bibitem{halko2011}
N.~Halko, P.G. Martinsson, and J.A. Tropp.
\newblock {Finding structure with randomness: Probabilistic algorithms for constructing approximate matrix decompositions}.
\newblock {\em {SIAM review}}, 53(2):217--288, 2011.

\bibitem{Higham}
N.~J. Higham.
\newblock {\em {Accuracy and Stability of Numerical Algorithms}}.
\newblock {SIAM}, {Philadelphia, PA, USA}, {second~ed.} edition, 2002.

\bibitem{Ein}
D.~Hilbert.
\newblock {Ein Beitrag zur Theorie des Legendre'schen Polynoms}.
\newblock {\em {Acta Mathematica}}, 18(none):155 -- 159, 1900.

\bibitem{2010}
M.~Hoemmen.
\newblock {\em {Communication-avoiding Krylov subspace methods}}.
\newblock {University of California, Berkeley}, 2010.

\bibitem{Improved}
C.P. Jeannerod and S.M. Rump.
\newblock {Improved error bounds for inner products in floating-point arithmetic}.
\newblock {\em {SIAM Journal on Matrix Analysis and Applications}}, 34:338--344, 2013.

\bibitem{Li}
Z.~Li, Y.~Wang, and S.~Li.
\newblock {The inverse eigenvalue problem for generalized Jacobi matrices with functional relationship}.
\newblock {\em {2015 12th International Computer Conference on Wavelet Active Media Technology and Information Processing (ICCWAMTIP)}}, pages 473--475, 2015.

\bibitem{2020}
P.G. Martinsson and J.A. Tropp.
\newblock {Randomized numerical linear algebra: Foundations and algorithms}.
\newblock {\em {Acta Numerica}}, 29:403 -- 572, 2020.

\bibitem{JCP}
D.P. O'Leary and G.W. Stewart.
\newblock {Computing the eigenvalues and eigenvectors of symmetric arrowhead matrices}.
\newblock {\em {Journal of Computational Physics}}, 90(2):497--505, 1990.

\bibitem{Eigen}
J.~Peng, X.~Hu, and L.~Zhang.
\newblock {Two inverse eigenvalue problems for a special kind of matrices}.
\newblock {\em {Linear Algebra and its Applications}}, 416:336--347, 2006.

\bibitem{Numerical}
M.~Rozložník, M.~Tůma, A.~Smoktunowicz, and J.~Kopal.
\newblock {Numerical stability of orthogonalization methods with a non-standard inner product}.
\newblock {\em {BIT Numerical Mathematics}}, pages 1--24, 2012.

\bibitem{Backward}
S.M. Rump and C.P. Jeannerod.
\newblock {Improved backward error bounds for LU and Cholesky factorization}.
\newblock {\em {SIAM Journal on Matrix Analysis and Applications}}, 35:684--698, 2014.

\bibitem{Super}
S.M. Rump and T.~Ogita.
\newblock {Super-fast vallidated solution of linear systems}.
\newblock {\em {Journal of Computational and Applied Mathematics}}, 199:199--206, 2007.

\bibitem{1989}
R.~Schreiber and C.~Van~Loan.
\newblock {A storage-efficient WY representation for products of Householder transformations}.
\newblock {\em {SIAM Journal on Scientific and Statistical Computing}}, 10(1):53--57, 1989.

\bibitem{Perturbation}
G.W. Stewart and J.~Sun.
\newblock {\em {Matrix perturbation theory}}.
\newblock {Academic Press}, {San Diego, CA, USA}, {sixth~ed.} edition, 1990.

\bibitem{Accurate}
N.J. Stor, I.~Slapničar, and J.L. Barlow.
\newblock {Accurate eigenvalue decomposition of real symmetric arrowhead matrices and applications}.
\newblock {\em {Linear Algebra and its Applications}}, 464:62--89, 2015.
\newblock {Special issue on eigenvalue problems}.

\bibitem{LUChol}
T.~Terao, K.~Ozaki, and T.~Ogita.
\newblock {LU-Cholesky QR algorithms for thin QR decomposition}.
\newblock {\em {Parallel Computing}}, 92:102571, 2020.

\bibitem{error}
Y.~Yamamoto, Y.~Nakatsukasa, Y.~Yanagisawa, and T.~Fukaya.
\newblock {Roundoff error analysis of the CholeskyQR2 algorithm}.
\newblock {\em {Electronic Transactions on Numerical Analysis}}, 44(01), 2015.

\bibitem{2016}
Y.~Yamamoto, Y.~Nakatsukasa, Y.~Yanagisawa, and T.~Fukaya.
\newblock {Roundoff error analysis of the CholeskyQR2 algorithm in an oblique inner product}.
\newblock {\em {JSIAM Letters}}, 8:5--8, 2016.

\bibitem{Mixed}
I.~Yamasaki, S.~Tomov, and J.~Dongarra.
\newblock {Mixed-precision Cholesky QR factorization and its case studies on Multicore CPU with Multiple GPUs}.
\newblock {\em {SIAM Journal on Scientific Computing}}, 37:{C307--C330}, 2015.

\bibitem{Modified}
Y.~Yanagisawa, T.~Ogita, and S.~Oishi.
\newblock {A modified algorithm for accurate inverse Cholesky factorization}.
\newblock {\em {Nonlinear Theory and Its Applications, IEICE}}, 5:35--46, 2014.

\end{thebibliography}

\end{document}